\documentclass[11pt]{article} 
\pdfoutput=1
\usepackage{amsfonts,amsthm,amssymb,amsmath,mathrsfs} 
\usepackage{pgfplotstable} 
\usepackage{graphicx,multirow} 
\usepackage{booktabs} 
\usepackage{algorithm} 
\usepackage{algorithmic}
\usepackage{tikz} 
\usepackage{url} 
\usepackage{enumerate}
\usepackage{makecell,hhline}
\usepackage{stmaryrd}
\usepackage{colortbl}

\usepackage{rotating} 
\usepackage{textcomp} 
\usepackage{pgfplots} 
\usepackage{xcolor} 
\usepackage{appendix}  
\usepackage{lipsum} 
\usepackage{subcaption} 
\usepackage{array}
\usepackage{tabu}
\usepackage{hyperref}
\newcommand{\thickhline}{%
    \noalign {\ifnum 0=`}\fi \hrule height 1pt
    \futurelet \reserved@a \@xhline
}
\newcolumntype{"}{@{\hskip\tabcolsep\vrule width 1pt\hskip\tabcolsep}}
 
\usepackage{siunitx}
 
\usepackage{tikz}  
\usepackage{pgfplots} 
 
\definecolor{blues1}{RGB}{198, 219, 239} 
\definecolor{blues2}{RGB}{158, 202, 225} 
\definecolor{blues3}{RGB}{107, 174, 214} 
\definecolor{blues4}{RGB}{49, 130, 189} 
\definecolor{blues5}{RGB}{8, 81, 156} 
\definecolor{blues6}{RGB}{2, 34, 78} 
 
\pgfplotscreateplotcyclelist{mylist}{ 
{blues1}, 
{blues2}, 
{blues3}, 
{blues4}, 
{blues5}, 
{blues6}, 
} 
 
\usepackage[top=1in,bottom=1in,left=1in,right=1in]{geometry} 
\numberwithin{equation}{section} 
\numberwithin{table}{section} 
\numberwithin{figure}{section} 
 
\newtheorem{definition}{Definition}[] 
\newtheorem{theorem}{Theorem}[]

\newtheorem{corollary}{Corollary}[] 
\newtheorem{remark}{Remark}[section] 
 
\newtheorem{assumption}{Assumption}[]

\theoremstyle{definition}
\newtheorem{experiment}{Experiment}[]
 
\newcommand{\m}[1]{\mathbf{#1}}
\newcommand{\R}{\mathbb{R}} 
\newcommand{\N}{\mathbb{N}} 
 
\newcommand{\bsx}{\boldsymbol{x}}
\newcommand{\bsy}{\boldsymbol{y}}
 
\newcommand{\1}{\mathbf{1}} 
 
\pgfplotsset{compat=1.18}
\begin{document}

\title{Efficient sparse probability measures recovery via Bregman gradient}


\author{Jianting Pan         \and
        Ming Yan 
}




\maketitle

\begin{abstract}
    This paper presents an algorithm tailored for the efficient recovery of sparse probability measures incorporating $\ell_0$-sparse regularization within the probability simplex constraint. Employing the Bregman proximal gradient method, our algorithm achieves sparsity by explicitly solving underlying subproblems. We rigorously establish the convergence properties of the algorithm, showcasing its capacity to converge to a local minimum with a convergence rate of $O(1/k)$ under mild assumptions. To substantiate the efficacy of our algorithm, we conduct numerical experiments, offering a compelling demonstration of its efficiency in recovering sparse probability measures.
\end{abstract}

\section{Introduction}

In this paper, we focus on solving the following sparse optimization problem with the $\ell_0$ regularization and the probability simplex constraint:
\begin{equation}
\label{opt}
\begin{aligned}
     \min_{\boldsymbol{x}\in \R^n}\quad   &f(\boldsymbol{x}) + \lambda \|\boldsymbol{x}\|_0 \\
     \mbox{subject to} \quad  & {\mathbf{1}_n}^\top\boldsymbol{x} = 1, \boldsymbol{x} \geq 0, 
\end{aligned}
\end{equation}
where $f: \R^n \rightarrow (-\infty, \infty]$ is proper, continuously differentiable and convex, $\lambda > 0$ is a regularization parameter, ${\mathbf{1}_n} \in \R^n$ is the vector with all ones, and $\boldsymbol{x} \geq 0$ indicates that all elements in $\boldsymbol{x}$ are nonnegative. The $\ell_0$ ``norm'' of a vector $\boldsymbol{x}$ counts the number of nonzero elements in $\boldsymbol{x}$. This problem encompasses various applications, including sparse portfolio optimization \cite{bertsimas2022scalable,yin2015novel,cura2009particle} and sparse hyperspectral unmixing \cite{salehani2014sparse,tang2014sparse,esmaeili2016l,majumdar2016impulse,rogass2014reduction,zou2018restoration}. For hyperspectral unmixing, the loss function $f(x)$ depends on the type of noise, which could be Gaussian noise~\cite{salehani2014sparse,tang2014sparse,esmaeili2016l}, impulse noise~\cite{majumdar2016impulse}, stripe noise~\cite{rogass2014reduction}, or Poisson noise~\cite{zou2018restoration}. 
Examples of $f(x)$ include the quadratic loss~\cite{salehani2014sparse} for the Gaussian noise, the Huber loss~\cite{guo2021modified} for the impulse noise, and the KL divergence for the Poisson noise~\cite{zou2018restoration}.

Various approaches are available for solving optimization problems with the $\ell_0$ term. The iterative hard-thresholding (IHT) algorithm was proposed for $\ell_0$-regularized least squares problems~\cite{blumensath2008iterative,blumensath2009iterative}. When the simplex constraint is incorporated, the paper~\cite{zhang2019learning} proposed an algorithm based on IHT and established its convergence properties to learn sparse probability measures. However, algorithms based on IHT require strong assumptions, such as mutual coherence~\cite{fornasier2015compressive} and restricted isometry condition~\cite{blumensath2009iterative}.

Due to the NP-hard nature of the $\ell_0$ term~\cite{natarajan1995sparse}, computationally feasible methods based on the $\ell_1$ norm, e.g., Lasso \cite{tibshirani1996regression}, have been introduced for problems without the simplex constraint. Bioucas-Dias and Figueiredo~\cite{bioucas2010alternating} proposed the SUnSAL algorithm, which applies the alternating direction method of multipliers (ADMM) to solve the following $\ell_1$ regularized problem: 
\begin{equation}
\begin{aligned}
\label{opt: SUnSAL}
    \min_{\boldsymbol{x}\in\R^n} \frac{1}{2}\|\m{A} \boldsymbol{x}-\boldsymbol{b}\|^2 + \lambda \|\boldsymbol{x}\|_1\quad \text{ subject to } ~ \boldsymbol{x}\geq 0,
\end{aligned}
\end{equation}
Note, the simplex constraint may not be satisfied here.


Other alternative terms were used besides the $\ell_0$ and $\ell_1$ terms. E.g., an iteratively reweighted algorithm based on the logarithm smoothed function was proposed in \cite{tang2014sparse}. The paper~\cite{esmaeili2016l} presented an algorithm based on ADMM to solve the following problem:
\begin{equation}
    \min_{\boldsymbol{x}\in\R^n} \frac{1}{2}\|\m{A} \boldsymbol{x}-\boldsymbol{b}\|^2 + \lambda F(\sigma,\boldsymbol{x})\quad \text{ subject to } { \mathbf{1}_n}^\top\boldsymbol{x}=1,~ \boldsymbol{x}\geq 0,
\end{equation}
where $\m{A} \in \mathbb{R}^{m \times n},~\boldsymbol{b} \in \R^m$,~$\lambda >0$ is a regularized parameter, and $F(\sigma,\boldsymbol{x}) = g(\sigma) \sum\limits_{i=1}^n \arctan (\sigma x_i)$ with $\sigma > 0$. The function $g(\sigma)$ is chosen such that $F(\sigma,\boldsymbol{x})$ tends to $\|\boldsymbol{x}\|_0$ as $\sigma \rightarrow \infty$.  
The paper~\cite{xiao2022geometric} used the $\ell_{1/2}$ norm regularization and solves the following equivalent problem:
\begin{equation}
    \begin{aligned}
         \min_{\boldsymbol{y}\in\R^n}   \frac{1}{2}\|\m{A} (\boldsymbol{y} \odot \boldsymbol{y})-\boldsymbol{b}\|_2^2+ \lambda \|\boldsymbol{y}\|_1\quad \text{subject to }   \boldsymbol{y}^\top \boldsymbol{y} = 1,
    \end{aligned}
\end{equation}
where $\lambda > 0$ is a regularized parameter and the symbol `$\odot$' means the Hadamard product of two vectors. This paper introduced a geometric proximal gradient (GPG) method to solve the above problem. By considering the $\ell_1$ norm and the constraints in the proximal mapping, GPG is essentially a proximal gradient method.

Although approximation models offer computational advantages, they may not precisely capture the solution of the original $\ell_0$-based model~\cite{zhang2023sparse}. Notably, an increasing body of research based on the $\ell_0$ term has recently emerged and attracted significant attention due to their remarkable recovery properties. In this context, noteworthy contributions have been made, such as normalized IHT and improved IHT~\cite{blumensath2010normalized,pan2017convergent}. Furthermore, to expedite convergence rates, various second-order algorithms rooted in the $\ell_0$ term, incorporating Newton-type steps, have been proposed \cite{zhang2023sparse,zhao2022lagrange,zhou2021global}. Despite the NP-hardness of the problem, the utilization of the $\ell_0$ term still has gained prominence in the realm of selecting sparse features.

This paper employs the Bregman proximal gradient (BPG) method to solve~\eqref{opt} and provides its theoretical guarantee. One of the primary challenges in solving~\eqref{opt} lies in projecting the solution onto the probabilistic simplex set. To tackle this challenge, we leverage BPG, allowing fast iterations by designing a suitable Bregman divergence (such as relative entropy, detailed in Section~\ref{algorithm} or Itakura-Saito distance). This choice mitigates computational burdens and reduces per-iteration complexity, facilitating effective convergence. Instead of enforcing a fixed number of elements to be zero, as done in methods like IHT~\cite{blumensath2008iterative,blumensath2009iterative}, we add a $\ell_0$ term and give an explicit expression of the global solution of the subproblem in each BPG iteration. The number of nonzero elements in each iteration adjusts according to the current iteration and the regularization parameter, providing greater flexibility than methods with a fixed number of nonzero elements. We establish the global convergence of our proposed algorithm and prove that the generated sequence converges to a local minimizer with the rate $O(1/k)$. Furthermore, while prior research predominantly relied on smoothed $\ell_0$ term \cite{xiao2022geometric}, our numerical results demonstrate that our proposed algorithm can achieve more accurate outcomes within a shorter timeframe.


\textbf{Notation.} Through this paper, we use bold lower letters for vectors, bold capital letters for matrices, and regular lower letters for scalars. The regular letter with a subscript indicates the corresponding element of the vector, e.g., $x_1$ is the first element of the vector $\boldsymbol{x}$. Let $\R^{m \times n}$ be the set of all $m \times n$ real matrices and $\R^n$ be equipped with the Euclidean inner product $\langle \cdot, \cdot \rangle$. The symbol `$\odot$' represents the Hadamard product of two vectors. We denote $\|\cdot\|_p$ as the $\ell_p$ norm of a vector. For simplicity, we use $\|\cdot\|$ to denote the Euclidean norm. For any $\boldsymbol{x} = (x_1,x_2,...,x_n)^\top \in \R^n$ and any set $I$, let $|I|$ denote the number of the elements in the set $I$ and $\text{supp}(\boldsymbol{x}) := \{i \in [n] : x_i \neq 0\}$, where $[n] := \{1,2,...,n\}$.

\section{The Proposed Algorithm}
We first introduce the standard BPG in Subsection~\ref{algorithm}. When applying BPG to our problem~\eqref{opt} in Section~\ref{Our proposed algorithm}, we must solve a subproblem with the $\ell_0$ term. Then, we propose a method to solve the subproblem analytically in Subsection~\ref{analytical solution}. We show that our BPG algorithm for solving the problem~\eqref{opt} converges in a finite number of iterations in Subsection~\ref{Convergence Analysis}. 

\subsection{Introduction to the Bregman proximal gradient}
\label{algorithm}
The Bregman proximal gradient (BPG) method, also known as mirror descent (MD) \cite{bauschke2017descent,eckstein1993nonlinear,bolte2018first,ben2001ordered,beck2003mirror,nemirovskij1983problem}, solves the following optimization problem
\begin{equation}
    \min_{\boldsymbol{x} \in C}\ f(\boldsymbol{x}),
\end{equation}
where $C$ is a closed convex set and the objective function $f$ is proper and continuously differentiable.

Let $h$ be a strictly convex function that is differentiable on an open set containing the relative interior of $C$~\cite{boyd2004convex}, which is denoted as $\mbox{rint}(C)$. For $\boldsymbol{y}\in \mbox{rint}(C)$, the Bregman divergence generated by $h$ is defined as 
\begin{equation}
    D_h(\boldsymbol{x},\boldsymbol{y}) = h(\boldsymbol{x}) - h(\boldsymbol{y}) - \langle \nabla h(\boldsymbol{y}), \boldsymbol{x}-\boldsymbol{y} \rangle,
\end{equation}
where $\boldsymbol{x}\in \text{dom}(h)$.
\begin{definition}
\label{Lsmooth}
    The function $f$ is called $L$-smooth relative to $h$ on $C$ if there exists $L>0$ such that, for $\boldsymbol{x}\in C$ and $\boldsymbol{y}\in \mbox{rint}(C)$,
    \begin{equation}
    \label{Lsmooth def}
        f(\boldsymbol{x}) \leq f(\boldsymbol{y})+\langle\nabla f(\boldsymbol{y}), \boldsymbol{x}-\boldsymbol{y}\rangle+L D_h(\boldsymbol{x}, \boldsymbol{y}).
    \end{equation}
\end{definition}
The definition of relative smoothness provides an upper bound for $f(\boldsymbol{x})$. If $f$ is $L$-smooth relative to $h$ on $C$, BPG updates the estimate of $\boldsymbol{x}$ via solving the following problems:
\begin{equation}
\label{general BPG}
    \boldsymbol{x}^{k+1} \in \arg \min_{\boldsymbol{x} \in C} \left( f(\boldsymbol{x}^k) + \langle \nabla f(\boldsymbol{x}^k),\boldsymbol{x}-\boldsymbol{x}^k \rangle + \frac{1}{\alpha}D_h(\boldsymbol{x},\boldsymbol{x}^k) \right),
\end{equation}
where $0 < \alpha  < 1/L$. 

The Bregman divergence generated by $h(\boldsymbol{x}) = \frac{1}{2}\|\boldsymbol{x}\|^2$ is the squared Euclidean distance $D_h(\boldsymbol{x},\boldsymbol{y}) = \frac{1}{2}\|\boldsymbol{x}-\boldsymbol{y}\|^2$, and the corresponding algorithm is the standard proximal gradient algorithm. A proper Bregman divergence can exploit optimization problems' structure~\cite{beck2003mirror} and reduce the per-iteration complexity. BPG has demonstrated numerous advantages in terms of computational efficiency in solving constrained optimization problems~\cite{bauschke2017descent,auslender2006interior,lu2018relatively,ben2001ordered,ma2011fixed,krichene2015accelerated,jiang2023bregman}. 

One of the most intriguing examples occurs when $C$ represents the probabilistic simplex set~\cite{ben2001ordered}. In this context, the proximal map becomes straightforward to compute when we utilize $h(\boldsymbol{x}) = \sum_{i=1}^n x_i \log x_i$ with the convention $0\log 0 = 0$ to generate the Bregman divergence. The Bregman divergence associated with such $h$ is
\begin{equation}
\label{KL divergence}
    D_h(\boldsymbol{x},\boldsymbol{y}) = \sum_{i=1}^n \left( x_i \log \left( \frac{x_i}{y_i} \right) - x_i + y_i \right),
\end{equation}
which is also known as KL-divergence or relative entropy. Under the simplex set constraint, the update \eqref{general BPG} admits the closed-form solution:
\begin{equation}
\label{closed form}
    x_i^{k+1} = \frac{x_i^k e^{-\alpha \nabla_{x_i} f(\boldsymbol{x}^k)}}{\sum_{j= 1}^n x_j^k e^{-\alpha \nabla_{x_j} f(\boldsymbol{x}^k)}}\quad  \forall\ i=1,2,...,n.
\end{equation}
The update of $\boldsymbol{x}^{k+1}$ in~\eqref{closed form} is much faster than the projection to the simplex set in the standard projected gradient descent. 

When $f$ is convex, BPG has a $O(1/k)$ convergence rate~\cite{bauschke2017descent,birnbaum2011distributed,lu2018relatively}. 
The paper~\cite{hanzely2021accelerated} proposed an accelerated Bregman proximal gradient method (ABPG) and ABPG with gain adaptation (ABPG-g), which have a faster convergence rate than BPG. Algorithm \ref{ABPG} presents the ABPG-g algorithm. According to~\cite{hanzely2021accelerated}, when applied under the probabilistic simplex set constraint with KL-divergence as the Bregman divergence and worked with intrinsic triangle scaling exponent $\gamma = 2$~\cite[Definition 3]{hanzely2021accelerated}, ABPG-g demonstrates an empirical convergence rate of $O(1/k^{2})$.

\begin{algorithm}
    \caption{ABPG with gain adaptation (ABPG-g)} 
    \label{ABPG} 
    \begin{algorithmic}
        \REQUIRE $\boldsymbol{z}^0 = \boldsymbol{x}^0 \in C,\ \gamma >1,\ \rho>1,\ \theta_0 = 1,\ G_{-1} = 1$, $G_{\text{min}} > 0$, $k = 0$, and $\varepsilon_1>0$.
        \REPEAT
        \STATE $G_k = \max\{G_{k-1}/\rho, G_{\text{min}}\}$
        \REPEAT
        \IF{$k > 0$} 
        \STATE compute $\theta_k$ by solving $\frac{1-\theta_k}{G_k \theta_k^{\gamma}} = \frac{1}{G_{k-1}\theta_{k-1}^{\gamma}}$. 
        \ENDIF
        \STATE $\boldsymbol{y}^k = (1-\theta_k)\boldsymbol{x}^k + \theta_k \boldsymbol{z}^k$
        \STATE $\boldsymbol{z}^{k+1} = \arg \min\limits_{\boldsymbol{z} \in C}\left\{f(\boldsymbol{y}^k) + \langle \nabla f(\boldsymbol{y}^k),\boldsymbol{z} - \boldsymbol{y}^k \rangle + G_k \theta_k^{\gamma-1} LD_h (\boldsymbol{z},\boldsymbol{z}^k)\right\}$
        \STATE $\boldsymbol{x}^{k+1} = (1-\theta_k)\boldsymbol{x}^k + \theta_k \boldsymbol{z}^{k+1}$
        \STATE $G_k \mapsfrom G_k \rho$
        \UNTIL{$f(\boldsymbol{x}^{k+1}) \leq f(\boldsymbol{y}^k) + \langle \nabla f(\boldsymbol{y}^k),\boldsymbol{x}^{k+1} - \boldsymbol{y}^k \rangle + G_k \theta_k^{\gamma} LD_h (\boldsymbol{z}^{k+1},\boldsymbol{z}^k)$}
        \STATE $k \mapsfrom k + 1$
        \UNTIL{$|f(\bsx^{k+1}) - f(\bsx^{k})| < \varepsilon_1$}
    \end{algorithmic} 
\end{algorithm}

\subsection{Applying BPG to problem~\eqref{opt}}
\label{Our proposed algorithm}

To begin with, it is important to emphasize that the iterates~\eqref{closed form} generated by BPG never reside on the boundary, i.e., ${x}^k_i \neq 0$ as long as the initial ${x}^0_i\neq 0$ for any index $i \in [n]$. However, these iterates may converge to the boundary without a sparse penalty term. We let the $\ell_0$ term be the penalty term to achieve sparsity during the iteration. Extensive research has demonstrated the effectiveness of the $\ell_0$ term in driving the iterates towards sparse solutions. Notably, the $\ell_0$ term exhibits stronger sparsity characteristics than alternative terms~\cite{xu2011image}, motivating us to employ it. In conclusion, by incorporating the $\ell_0$ term into the BPG algorithm, we aim to solve the following subproblem:
$$
\boldsymbol{x}^{k+1} \in \arg \min_{{\mathbf{1}_n}^\top\boldsymbol{x} = 1} \left( f(\boldsymbol{x}^k) + \langle \nabla f(\boldsymbol{x}^k),\boldsymbol{x}-\boldsymbol{x}^k \rangle + \frac{1}{\alpha}D_h(\boldsymbol{x},\boldsymbol{x}^k) + \lambda \|\boldsymbol{x}\|_0\right),
$$
where $\boldsymbol{x}^k$ is the current iterate, $D_h(\bsx,\bsx^k)$ is the KL-divergence between $\bsx$ and $\bsx^k$, $\alpha>0$ and $\lambda>0$ are two constants. 
Based on the above analysis, our algorithm for solving problem \eqref{opt} can be described as follows in Algorithm~\ref{alg}.

\begin{algorithm}
	\caption{Our proposed algorithm for~\eqref{opt}} 
	\label{alg} 
	\begin{algorithmic}
		\REQUIRE $\lambda, \varepsilon_2$, $\alpha > 0$, $k = 0$.
        \STATE \textbf{(Initialization)}: Use \textbf{ABPG-g} (Algorithm \ref{ABPG}) to attain $\boldsymbol{x}^0$ (note: $\|\boldsymbol{x}^0\|_0 = n$).
        \REPEAT
        \STATE \textbf{(L0BPG)}: Update $\boldsymbol{x}^{k+1}$ via
        \begin{equation}
        \label{subproblem}
        \boldsymbol{x}^{k+1} \in \arg \min_{{\1_n^\top}\boldsymbol{x} = 1} \left( f(\boldsymbol{x}^k) + \langle \nabla f(\boldsymbol{x}^k),\boldsymbol{x}-\boldsymbol{x}^k \rangle + \frac{1}{\alpha}D_h(\boldsymbol{x},\boldsymbol{x}^k) + \lambda \|\boldsymbol{x}\|_0\right).
        \end{equation}
        \STATE $k \mapsfrom k + 1$
        \UNTIL{$f(\bsx^k) + \lambda \|\bsx^k\|_0 - f(\bsx^{k+1}) - \lambda \|\bsx^{k+1}\|_0 < \varepsilon_2$}
	\end{algorithmic} 
\end{algorithm}


\begin{remark}
\label{initialization phase}
At the initialization phase of Algorithm \ref{alg}, we utilize ABPG-g to obtain a proper starting point $\boldsymbol{x}^0$. As demonstrated in \cite{hanzely2021accelerated}, ABPG-g exhibits an empirical convergence rate of $O(1/k^{2})$, which is notably faster than the convergence rate of $O(k^{-1})$ observed in BPG~\cite{birnbaum2011distributed,bauschke2017descent}. Consequently, we can establish an appropriate starting point via ABPG-g more expeditiously than BPG. Furthermore, it's important to note that we do not achieve a sparse solution during this initialization phase, i.e., $\|\boldsymbol{x}^0\|_0 = n$. 
\end{remark}

\subsection{Analytical solution to the subproblem~\eqref{subproblem}}
\label{analytical solution}

The following theorem provides a way to find a global solution to the subproblem \eqref{subproblem}. 

\begin{theorem}
\label{global solution}
Let 
\begin{equation}
\begin{aligned}
\label{yk+1}
\boldsymbol{y}^{k+1} =\arg \min_{{\mathbf{1}_n}^\top \boldsymbol{x} = 1} \left( f(\boldsymbol{x}^k) + \langle \nabla f(\boldsymbol{x}^k),\boldsymbol{x}-\boldsymbol{x}^k \rangle + \frac{1}{\alpha}D_h(\boldsymbol{x},\boldsymbol{x}^k) \right)  
\end{aligned}
\end{equation}
and
$$
d \in \arg \min_{m \in [n]} -\frac{1}{\alpha}\log\sum_{i=1}^m y^{k+1}_{(i)} + \lambda m,
$$
where $y^{k+1}_{(i)}$ represents the $i$-th largest element of $\boldsymbol{y}^{k+1}$,i.e. $y^{k+1}_{(1)} \geq y^{k+1}_{(2)} \geq ... \geq y^{k+1}_{(n)}$, then we obtain a global solution to the subproblem \eqref{subproblem} as
\begin{equation}
\label{zk+1}
\begin{aligned}
x_i^{k+1} = \left\{
    \begin{aligned}
    & \frac{y^{k+1}_i}{\sum_{j \in I_{k+1}} y^{k+1}_j}, & \quad \mbox{if }i \in I_{k+1},\\
    & 0, & \quad otherwise,
    \end{aligned}
\right.
\end{aligned}
\end{equation}
where $I_{k+1}$ is the set of the indices of the first $d$ largest entries of $\boldsymbol{y}^{k+1}$. 
\end{theorem}
\begin{proof}
Denote $g(x_i) = x_i \nabla_{x_i}f(\boldsymbol{x}^k) + \frac{x_i}{\alpha}\log\left(\frac{x_i}{x_i^k}\right)$. Then the subproblem \eqref{subproblem} can be rewritten as
\begin{equation}
\label{two-layer}
\begin{aligned}
    & \min_{{\mathbf{1}_n}^\top \boldsymbol{x} = 1}  f(\boldsymbol{x}^k) + \langle \nabla f(\boldsymbol{x}^k),\boldsymbol{x}-\boldsymbol{x}^k \rangle + \frac{1}{\alpha}D_h(\boldsymbol{x},\boldsymbol{x}^k) + \lambda \|\boldsymbol{x}\|_0\\
    = & \min_{{\mathbf{1}_n}^\top \boldsymbol{x} = 1 }\sum_{i=1}^n g(x_i) + \lambda \|\boldsymbol{x}\|_0 + c \\
    = & \min_{m \in [n]}\ \min_{{\mathbf{1}_n}^\top \boldsymbol{x} = 1, \|\boldsymbol{x}\|_0 = m} \sum_{i=1}^n g(x_i) + \lambda m + c\\
    = & \min_{m \in [n]}\  \min_{{\mathbf{1}_n}^\top \boldsymbol{x} = 1, \|\boldsymbol{x}\|_0 = m} \sum_{i \in \text{supp}(\boldsymbol{x})} g(x_i) + \lambda m + c,
\end{aligned}
\end{equation}
where $c=f(\boldsymbol{x}^k)-\langle\nabla f(\boldsymbol{x}^k),\boldsymbol{x}^k\rangle$. The last equality holds since $g(0) = 0$. 

Let's consider the inner minimization problem first
\begin{equation}
\label{min1}
 \min_{{\mathbf{1}_n}^\top \boldsymbol{x} = 1, \|\boldsymbol{x}\|_0 = m} \sum_{i \in \text{supp}(\boldsymbol{x})} g(x_i)= \min_{|I| = m}\quad \min_{{\mathbf{1}_n}^\top \boldsymbol{x} = 1,~\text{supp}(\boldsymbol{x})= I} \sum_{i \in I} g(x_i).
\end{equation}
There are $C_{n}^{m}$ possible ways to choose the support $I$ of $m$ elements from the $n$ elements. For each fixed support of $m$ elements, we can find the optimal $\boldsymbol{x}$ analytically. Then, the problem becomes finding the support of $m$ elements with the smallest function value from those possible ways. 
Given the support $I$ of $m$ elements for $\boldsymbol{x}$, we can solve the problem analytically as below:
\begin{equation}
\label{possible xk+1}
x_{m,i}^{k+1} = \left\{
    \begin{aligned}
    & \frac{x_i^k e^{-\alpha \nabla_{x_i} f(\boldsymbol{x}^k)}}{\sum_{j \in I} x_j^k e^{-\alpha \nabla_{x_j} f(\boldsymbol{x}^k)}}, & \quad \mbox{if } i \in I,\\
    & 0, & \quad \mbox{otherwise.}
    \end{aligned}\right.
\end{equation}
We plug this solution into the objective function in~\eqref{min1} and obtain 
\begin{align*}
    \sum_{i \in I} g(x_{m,i}^{k+1})=-\frac{1}{\alpha}\log  \sum_{i \in I} x_i^k e^{-\alpha \nabla_{x_i} f(\boldsymbol{x}^k)}.
\end{align*}

Note that the optimization problem in~\eqref{yk+1} gives that 
$$
{y}_i^{k+1}=\frac{{x_i}^k e^{-\alpha \nabla_{x_i} f(\boldsymbol{x}^k)}}{\sum_{j=1}^n{x}^k_j  e^{-\alpha \nabla_{x_j} f(\boldsymbol{x}^k)}}.
$$
Therefore, the objective function in~\eqref{min1} becomes 
\begin{align*}
    \sum_{i \in I} g(x_{m,i}^{k+1})=-\frac{1}{\alpha}\log  \sum_{i \in I} y_i^{k+1}-\frac{1}{\alpha}\log \sum_{j=1}^n{x}^k_j  e^{-\alpha \nabla_{x_j} f(\boldsymbol{x}^k)}.
\end{align*}
Thus, we must choose the indices for the $m$ largest elements from $\boldsymbol{y}^{k+1}$. 

Since the inner optimization problem in~\eqref{two-layer} can be solved analytically, the subproblem~\eqref{subproblem} reduces to finding the number $m$ by solving the problem 
\begin{align}
    \min_{m \in [n]}\ -\frac{1}{\alpha}\log  \sum_{i=1}^m {y}^{k+1}_{(i)}+ \lambda m  -\frac{1}{\alpha}\log \sum_{j=1}^n{x}^k_j  e^{-\alpha \nabla_{x_j} f(\boldsymbol{x}^k)}+ c,
\end{align}
which is equivalent to 
$$d \in \arg \min_{m \in [n]} -\frac{1}{\alpha}\log\sum_{i=1}^m {y}^{k+1}_{(i)} + \lambda m.$$
After we find the number $d$, we choose the indices as the largest $d$ elements from $\boldsymbol{y}^{k+1}$, then we construct $\boldsymbol{x}^{k+1}$ based on the equation~\eqref{zk+1}. 
\end{proof}

Based on Theorem \ref{global solution}, we can solve the problem \eqref{subproblem} by Algorithm \ref{solve subproblem}. Given that there may be two choices for $d$, we opt to select the larger one, i.e.,
\begin{equation}
    d_{k+1} := \max \left\{ \arg \min_{m \in [n]} -\frac{1}{\alpha}\log\sum_{i=1}^m {y}^{k+1}_{(i)} + \lambda m\right\}.
\end{equation}

\begin{algorithm}
	\caption{Algorithm to solve the problem~\eqref{subproblem}} 
	\label{solve subproblem} 
	\begin{algorithmic}
        \STATE \textbf{(BPG step)}: Update $\boldsymbol{y}^{k+1}$ via
        \begin{equation}
        \label{BPG}
        \boldsymbol{y}^{k+1} = \arg \min_{{\mathbf{1}_n}^\top \boldsymbol{y} = 1} \left( f(\boldsymbol{x}^k) + \langle \nabla f(\boldsymbol{x}^k),\boldsymbol{y}-\boldsymbol{x}^k \rangle + \frac{1}{\alpha}D_h(\boldsymbol{y},\boldsymbol{x}^k) \right).
        \end{equation}
        \STATE \textbf{(Sorting step)}: Find $d_{k+1}$ such that
        \begin{equation}
        \label{sorting}
        d_{k+1} = \max \left\{\arg \min_{m \in [n]} -\frac{1}{\alpha}\log\sum_{i=1}^m {y}^{k+1}_{(i)} + \lambda m \right\},
        \end{equation}
        where we order the elements of $\boldsymbol{y}^{k+1}: y_{(1)}^{k+1} \geq y_{(2)}^{k+1} \geq ... \geq y_{(n)}^{k+1}$. 
        \STATE \textbf{(Removing step)}: Update $\boldsymbol{x}^{k+1}$ by
        \begin{equation}
        \label{removing}
            x_i^{k+1} = \left\{
                \begin{aligned}
                & \frac{y^{k+1}_i}{\sum_{j \in I_{k+1}} y^{k+1}_j}, & \quad \text{for }i \in I_{k+1},\\
                & 0, & \quad \text{otherwise},
                \end{aligned}
            \right.
        \end{equation}
        where $I_{k+1}$ is the set of the indices of the first $d_{k+1}$ largest entries of $\boldsymbol{y}^{k+1}$. 
	\end{algorithmic} 
\end{algorithm}

\begin{remark}
    From the sorting step~\eqref{sorting} and the removing step~\eqref{removing}, it is evident that $|I_{k+1}| = d_{k+1}$. Moreover, based on the closed-form expression of $\bsy^{k+1}$, it is straightforward to observe that the sequence $\{d_k\}_{k \in \N}$ is nonincreasing and $I_{k+1} \subseteq I_k$. Hence, once $x^{k+1}_i$ becomes $0$ for some $i \in [n]$ in the removing step~\eqref{removing}, it cannot be positive again. 
\end{remark}

\begin{remark}
    It also indicates the importance of the initialization phase of Algorithm~\ref{alg} with high accuracy, i.e. small $\varepsilon_1 (=10^{-6},10^{-7})$. Firstly, high accuracy would be more likely to preserve important elements. With high precision, the elements in $\bsx$ change very little. Most elements will be close to $0$ and not be in the ground truth support set $I^*$. Therefore, setting these elements to $0$ will not affect our search for the support set. Secondly, high accuracy can help accelerate the convergence. It would set many insignificant elements to $0$ at the first removing step. Hence, we attain a much lower dimensional optimization problem and accelerate the convergence. We also emphasize that the more complex (heavier noise) the problem is, the higher the accuracy is needed.
\end{remark}
Let's denote 
\begin{equation}
l(m) = -\frac{1}{\alpha}\log\sum_{i=1}^m y^{k+1}_{(i)} + \lambda m. 
\end{equation}
In the sorting step \eqref{sorting}, calculating $l(m)$ from $m = 1$ to $m = n$ can be time-consuming. However, it is unnecessary to compute $l(m)$ for all $n$ values because the following theorem shows that $l(m)$ decreases first and then increases when $m$ increases from 1 to $n$.


\begin{theorem}
\label{lambda bound}
    In the sorting step \eqref{sorting}, $l(m)$ is monotonically decreasing for $m \in \{1,2,...,d_{k+1}-1\}$, and monotonically increasing for $m \in \{d_{k+1},d_{k+1}+1,...,n\}$. 
\end{theorem}

\begin{proof}
    We check the difference between two successive values $l(m+1)-l(m)$. For $m=1,\dots,n-1$, we have 
    \begin{align*}l(m+1)-l(m)=&-\frac{1}{\alpha}\log\sum_{i=1}^{m+1} y^{k+1}_{(i)}+\frac{1}{\alpha}\log\sum_{i=1}^m y^{k+1}_{(i)}+\lambda\\
    =&-{1\over \alpha}\log\left(1+{y^{k+1}_{(m+1)}\over \sum_{i=1}^{m} y^{k+1}_{(i)}}\right)+\lambda.
    \end{align*}
    Since $y^{k+1}_{(m)}$ is nonincreasing, we have that $l(m+1)-l(m)$ increases as $m$ increases from $1$ to $\hat{k}$, where $\hat{k}$ is the number such that $y_{(\hat{k})}^{k+1} = 0$. Therefore, in the sorting step~\eqref{sorting}, we let the smallest $m$ such that $l(m+1)-l(m)$ is positive be the solution $d_{k+1}$. In this case, $l(m)$ is increasing for $m\geq d_{k+1}$. Note that we have $l(d_{k+1})-l(d_{k+1}-1)\leq 0$ and $l(d_{k+1}-1)-l(d_{k+1}-2)<0$, therefore, $l(m)$ is decreasing for $m\leq d_{k+1}-1$. 
    It could happen that $l(d_{k+1})=l(d_{k+1}-1)$, and in this case, we choose the larger number $d_{k+1}$ as we mentioned in Algorithm~\ref{solve subproblem}.   
\end{proof}


Based on the previous theorem, we determine $d_{k+1}$ as the smallest $m$ such that $l(m+1)>l(m)$, that is
\begin{equation}\label{findm}
    e^{\alpha \lambda} -1> \frac{ y_{(m+1)}^{k+1}}{\sum_{i=1}^m y_{(i)}^{k+1}}=\frac{ y_{(m+1)}^{k+1}}{1-\sum_{i=m+1}^n y_{(i)}^{k+1}}.
\end{equation}
We can choose to check the inequality starting from $m=1$ or $m=d_k$ depending on $\boldsymbol{y}^{k+1}$ values.

\subsection{Convergence analysis of Algorithm~\ref{alg}}
\label{Convergence Analysis}
Throughout this subsection, we have the following assumption on $f$.
\begin{assumption}
\label{convex assumption}
    $f:\R^n \rightarrow (-\infty,\infty]$ is proper, continuously differentiable, and convex. In addition, $f$ is $L$-smooth relative to $h$. 
\end{assumption}

For simplicity, we denote $F(\boldsymbol{x}) = f(\boldsymbol{x})+\lambda \|\bsx\|_0$. 


\begin{theorem}\label{decrease thm}
    \textbf{(Descent property)} Under Assumption~\ref{convex assumption}, let $\{\boldsymbol{x}^k\}_{k \in \N}$ be the sequence generated by Algorithm~\ref{alg} with $0 < \alpha < 1/L$, then the sequence $\{F(\boldsymbol{x}^k)\}_{k \in \N}$ is nonincreasing and converges. The support $\{\textnormal{supp}\{\boldsymbol{x}^k\}\}_{k \in \N}$ converges in a finite number of iterations, i.e., $\exists M>0$ such that $\textnormal{supp}(\boldsymbol{x}^k) = I \subset [n]$ for $\forall k \geq M$. 
\end{theorem}

\begin{proof}
    Notice that 
    $$
    \begin{aligned}
        F(\boldsymbol{x}^k) & = f(\boldsymbol{x}^k) + \lambda \|\boldsymbol{x}^k\|_0 \\
        & = f(\boldsymbol{x}^k) + \langle \nabla f(\boldsymbol{x}^k),\boldsymbol{x}^{k}-\boldsymbol{x}^k \rangle + \frac{1}{\alpha} D_h(\boldsymbol{x}^{k},\boldsymbol{x}^k) + \lambda \|\boldsymbol{x}^k\|_0 \\
        & \geq f(\boldsymbol{x}^k) + \langle \nabla f(\boldsymbol{x}^k),\boldsymbol{x}^{k+1}-\boldsymbol{x}^k \rangle + \frac{1}{\alpha} D_h(\boldsymbol{x}^{k+1},\boldsymbol{x}^k) + \lambda \|\boldsymbol{x}^{k+1}\|_0 \\
        & \geq f(\boldsymbol{x}^{k+1}) + \lambda \|\boldsymbol{x}^{k+1}\|_0 \\
        & = F(\boldsymbol{x}^{k+1}).
    \end{aligned}
    $$
    The first inequality holds due to \eqref{subproblem}, and the second is by Assumption~\ref{convex assumption}. Since $\boldsymbol{x}^k \in [0,1]^n$, the sequence $\{\boldsymbol{x}^k\}_{k \in \N}$ is bounded. Hence, the sequence $\{F(\boldsymbol{x}^k)\}_{k \in \N}$ is bounded below and converges to a limit $F^*$, i.e., $\lim\limits_{k \rightarrow \infty} F(\boldsymbol{x}^k) = F^*$. In addition, the number of number elements $\{\|\boldsymbol{x}^k\|_0\}_{k \in \N}$ is nonincreasing and converges. 
    Thus, the support of $\boldsymbol{x}^k$ converges to a set $I\subset [n]$. 
\end{proof}

\begin{remark}
    The above theorem says that after finite iterations, $\{\text{supp}(\boldsymbol{x}^k)\}_{k \in \N}$ remains the same. The sorting step \eqref{sorting} and removing step \eqref{removing} are redundant, and $\boldsymbol{x}^k$ solves the following lower-dimension convex optimization problem using BPG:
    \begin{equation}
        \min_{{\mathbf{1}_n}^\top \boldsymbol{x} = 1,\ \boldsymbol{x} \geq 0} f(\boldsymbol{x}) \quad \text{subject to } x_i = 0, i \not\in I.
    \end{equation}
\end{remark}
In the remaining of this subsection, we let $I=\textnormal{supp}(\boldsymbol{x}^k)$ for large enough $k$ and denote the solution set of the optimization problem~\eqref{subproblem} as $X^*$, i.e., 
\begin{equation}
\label{X*}
X^* = \arg \min \{f(\boldsymbol{x}): {\mathbf{1}_n}^\top \boldsymbol{x} = 1, \boldsymbol{x} \geq 0\ \text{and}\ x_i = 0, i \not\in I\}.
\end{equation}
    
\begin{corollary}
\label{properties of BPG}
    Under Assumption~\ref{convex assumption}, let $\{\boldsymbol{x}^k\}_{k \in \N}$ be the sequence generated by Algorithm~\ref{alg} with $0 < \alpha < 1/L$, then
    \begin{enumerate}[i)]
        \item After finite iterations, we have
        \begin{equation}
        \label{decreasing BPG}
        \begin{aligned}
            \alpha \left( f(\boldsymbol{x}^{k+1}) - f(\boldsymbol{x}) \right) \leq & D_h(\boldsymbol{x},\boldsymbol{x}^k) - D_h(\boldsymbol{x},\boldsymbol{x}^{k+1})\\
            & - (1-\alpha L)D_h(\boldsymbol{x}^{k+1},\boldsymbol{x}^k),
        \end{aligned}
        \end{equation}
        for $\forall \boldsymbol{x} \in \{\boldsymbol{y}: {\1_n^\top} \boldsymbol{y} = 1, \boldsymbol{y} \geq 0, y_i = 0, i \not \in I\}$. 
        \item $D_h(\boldsymbol{x}^{k+1},\boldsymbol{x}^{k})$ converges to 0 as $k \rightarrow \infty$.
        \item The sequence $\{\boldsymbol{x}^k\}_{k \in \N}$ converges to some $\boldsymbol{x}^* \in X^*$. 
    \end{enumerate}
\end{corollary}

\begin{proof}
    By Theorem \ref{decrease thm}, after $M$ iterations, $\{\text{supp}(\boldsymbol{x}^k)\}_{k \in \N}=I$. The proof follows directly from \cite{bauschke2017descent}. 
\end{proof}


\begin{theorem}
\label{function value converge}
    Under Assumption~\ref{convex assumption}, the sequence $\{\boldsymbol{x}^k\}_{k \in \N}$ generated by Algorithm~\ref{alg} with $0 < \alpha  < 1/L$ converges to $\boldsymbol{x}^*\in X^*$ with $\textnormal{supp}(\boldsymbol{x}^*)=I$ and $x_i^* \geq 1- e^{-\alpha \lambda}$ $\forall i \in I$. In addition, $\boldsymbol{x}^*$ is a local minimum point of $F(\boldsymbol{x})$ over the simplex set $S: = \{\boldsymbol{x}: {\1_n^\top} \boldsymbol{x} = 1, \boldsymbol{x} \geq 0\}$. If $\textnormal{supp}(\boldsymbol{x}^M)=I$, then, for $K\geq M+1$, 
    \begin{equation}
        F(\boldsymbol{x}^{K}) - F(\boldsymbol{x}^*) \leq \frac{1}{\alpha(K-M)} D_h(\boldsymbol{x}^*,\boldsymbol{x}^{M}).
    \end{equation}   
\end{theorem}
\begin{proof}
    The global convergence of $\{\boldsymbol{x}^k\}_{k \in \N}$ comes from Corollary~\ref{properties of BPG}. Next, we show that $\textnormal{supp}(\boldsymbol{x}^*)=I$ and $\boldsymbol{x}^*$ is a local minimum point of $F(\boldsymbol{x})$.
 
    From Theorem~\ref{lambda bound}, we have $l(d_{k+1})-l(d_{k+1}-1)\leq 0$, which gives 
    \begin{align*}
         \frac{ y_{(d_{k+1})}^{k+1}}{\sum_{i=1}^{d_{k+1}} y_{(i)}^{k+1}}\geq 1-e^{-\alpha\lambda}.
    \end{align*}
    Then Theorem~\ref{global solution} shows that $x_i^{k+1}\geq 1-e^{-\alpha \lambda}$ for $i\in I_{k+1}$. Since $\boldsymbol{x}^k$ converges to $\boldsymbol{x}^*$, we have ${x}^*_i\geq 1-e^{-\alpha\lambda}$ for $i\in I$. Thus $\textnormal{supp}(\boldsymbol{x}^*)=I$ and $F(\boldsymbol{x}^k) \rightarrow F(\boldsymbol{x}^*)$ as $k \rightarrow \infty$. 
    
    Now we show that $\boldsymbol{x}^*$ is a local minimum point of $F(\boldsymbol{x})$ over the set $S$, i.e., there exist $\delta>0$ such that $F(\boldsymbol{x})>F(\boldsymbol{x}^*)$ for any $\boldsymbol{x}\in S$ such that $\|\boldsymbol{x} - \boldsymbol{x}^*\|<\delta$. 
    Denote $\delta_1 = 1-e^{-\alpha \lambda}$. Since $x_i^*\geq \delta_1$, for any $\boldsymbol{x}$ such that $\|\boldsymbol{x}-\boldsymbol{x}^*\| < \delta_1$, we have $\textnormal{supp}(\boldsymbol{x})\supset \textnormal{supp}(\boldsymbol{x^*})$. In addition, the continuity of $f$ shows that there exists $\delta_2 > 0$, such that $|f(\boldsymbol{x}) - f(\boldsymbol{x}^*)| < \lambda$ if  $\|\boldsymbol{x} - \boldsymbol{x}^*\| < \delta_2$. Let $\delta=\min\{\delta_1,\delta_2\}$, and we consider $\boldsymbol{x}$ such that $\|\boldsymbol{x} - \boldsymbol{x}^*\| < \delta$.    
    \begin{itemize}
        \item If $\text{supp}(\boldsymbol{x}) = \text{supp}(\boldsymbol{x}^*)$, Corollary~\ref{properties of BPG} shows that $\boldsymbol{x}^*\in X^*$, thus $F(\boldsymbol{x})\geq F(\boldsymbol{x}^*)$ if $\boldsymbol{x}\in S$.
        \item If $\text{supp}(\boldsymbol{x}) \supsetneqq  \text{supp}(\boldsymbol{x}^*)$, we have
        $$
        \begin{aligned}
            F(\boldsymbol{x}) &= f(\boldsymbol{x}) + \lambda \|\boldsymbol{x}\|_0 \\
            & > f(\boldsymbol{x}^*) - \lambda + \lambda \|\boldsymbol{x}\|_0 \\ 
            & \geq f(\boldsymbol{x}^*) + \lambda \|\boldsymbol{x}^*\|_0 = F(\boldsymbol{x}^*).
        \end{aligned}
        $$
    \end{itemize}
    Hence, $\boldsymbol{x}^*$ is a local minimum of $F$ over the set $S$. 

    If $\textnormal{supp}(\boldsymbol{x}^M)=I$, then we have $\textnormal{supp}(\boldsymbol{x}^k)=I$ for all $k\geq M$. 
    By i) in Corollary~\ref{properties of BPG}, for $k\geq M$ we have
    $$
    F(\boldsymbol{x}^{k+1}) - F(\boldsymbol{x}^*) = f(\boldsymbol{x}^{k+1}) - f(\boldsymbol{x}^*) \leq \frac{1}{\alpha} \left( D_h(\boldsymbol{x}^*,\boldsymbol{x}^k) - D_h(\boldsymbol{x}^*,\boldsymbol{x}^{k+1}) \right).
    $$
    Since $\{F(\boldsymbol{x}^{k})\}_{k \in \N}$ is nonincreasing, we have 
    $$
    \begin{aligned}
        F(\boldsymbol{x}^{K}) - F(\boldsymbol{x}^*) & \leq \frac{1}{K-M}\sum_{k = M}^{K-1} \left( F(\boldsymbol{x}^{k+1}) - F(\boldsymbol{x}^*) \right)\\
            & \leq \frac{1}{\alpha(K-M)} D_h(\boldsymbol{x}^*,\boldsymbol{x}^{M}).
    \end{aligned}
    $$
    The theorem is proved. 
\end{proof}

\begin{remark}
    The above theorem indicates that given $\lambda$ and $\alpha$, one can control the minimal value in $I$. 
\end{remark}

\section{Numerical Experiments}
In this section, we present the numerical performance of Algorithm \ref{alg} for solving the problem \eqref{opt}. For the initialization phase (Algorithm~\ref{ABPG}) in Algorithm~\ref{alg}, we follow the parameter setting in~\cite{hanzely2021accelerated} and set $\gamma = 2$, $\rho = 1.2$, and $G_{\text{min}} = 10^{-2}$ to obtain $\bsx^0$. The initial points for Algorithm~\ref{alg} and GPG~\cite{xiao2022geometric} are set to be $(1/n,...,1/n)^\top$. The required time of Algorithm~\ref{alg} includes both the initialization phase and L0BPG step. All numerical experiments are implemented by running MATLAB R2023b on a MacBook Pro (Apple M2 Pro). 



\subsection{Experiments on synthetic data}

\label{sec: synthetic}
In this subsection, we use the same setup as~\cite{xiao2022geometric}. We consider the data $\boldsymbol{b}$ generated by
$$
\boldsymbol{b} = \m{A} \boldsymbol{x}^* + \boldsymbol{n},
$$
where $\m{A} \in \R^{m \times n}$, whose elements are independently sampled from standard Gaussian distribution, and $\boldsymbol{n}$ is the noise. The original vector $\boldsymbol{x}^*$ is generated by $\boldsymbol{x}^* = |\Bar{\boldsymbol{x}}|/\|\Bar{\boldsymbol{x}}\|_1$, where $\Bar{\boldsymbol{x}} \in \R^n$ is generated using \texttt{sprandn} from Matlab and $|\bar{\boldsymbol{x}}|$ takes the element-wise absolute values of $\bsx$. In addition, to control the noise level, we define the signal-to-ratio (SNR)~\cite{xiao2022geometric} of data $\boldsymbol{b}$ as 
$$
SNR = 10\log_{10} \frac{\|\m{A} \boldsymbol{x}^*\|^2}{\|\boldsymbol{n}\|^2}.
$$

To find the original vector $\boldsymbol{x}^*$, we consider the following optimization problem:
\begin{equation}
\begin{aligned}
\label{opt: linear model}
    \min_{\boldsymbol{x} \in \R^{n}} \quad & \sum_{i=1}^m \phi(b_i - \boldsymbol{a}_i^\top \boldsymbol{x}) + \lambda \|\boldsymbol{x}\|_0\\
    \text{subject to } \quad  & \mathbf{1}_n^\top \boldsymbol{x} = 1,\ \boldsymbol{x} \geq 0,
\end{aligned}
\end{equation}
where $\boldsymbol{a}_i^\top$ is the $i$th row of $\m{A}$ and $\phi$ is the loss function. To evaluate the quality of the recovered vectors, we use reconstruction SNR (RSNR), which is defined as 
$$
RSNR = 10\log_{10} \frac{\|\boldsymbol{x}^*\|^2}{\|\boldsymbol{x}^* - \hat{\boldsymbol{x}}\|^2},
$$
where $\hat{\boldsymbol{x}}$ is the recovered vector. 

In Experiments~\ref{SNR test}-\ref{exp: efficacy L0BPG step}, we consider the Gaussian noise and $\phi$ is quadratic. In Experiment~\ref{exp: huber loss}, we consider mixed noise and adopt Huber loss in~\eqref{opt: linear model}.

\begin{figure}
    \centering
    \includegraphics[scale = 0.4]{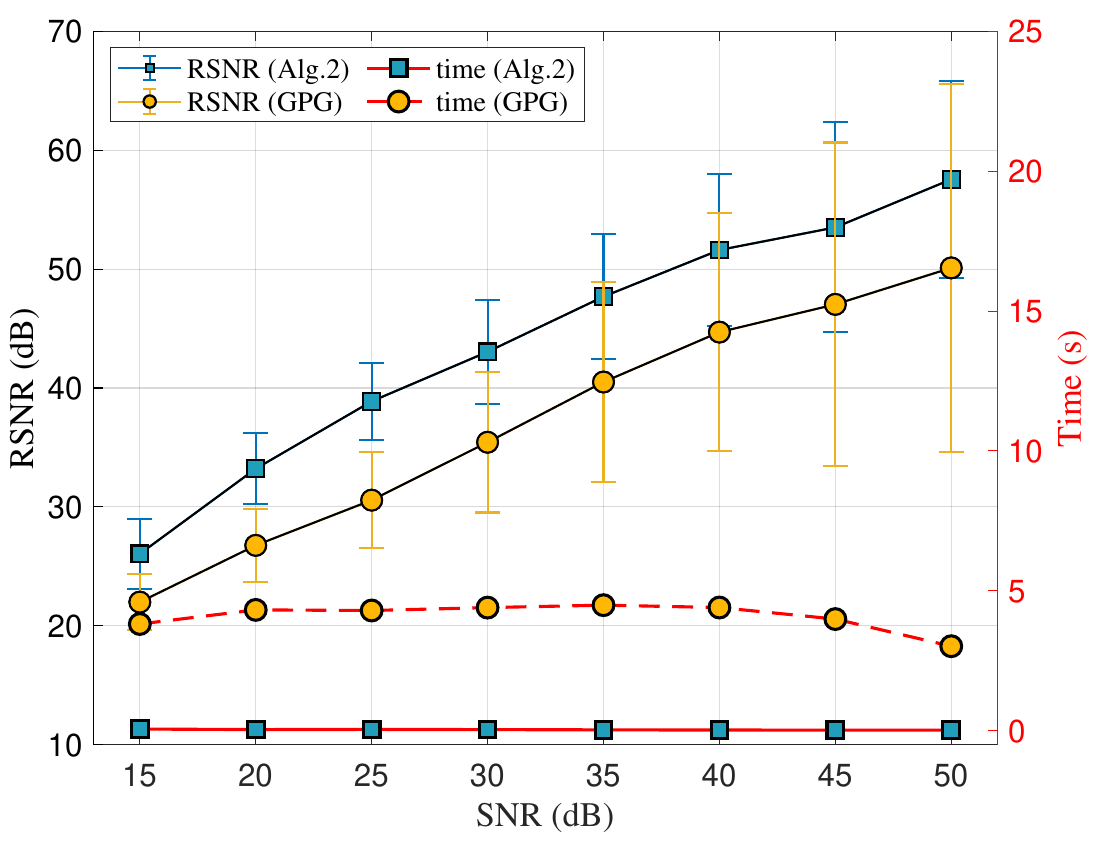}
    \caption{Comparison of Algorithm~\ref{alg} and GPG in accuracy (RSNR) and time for different SNRs (averaged over 100 runs).}
    \label{SNR result}
\end{figure}

\begin{experiment}
\label{SNR test}
\textbf{(recovery accuracy)} We set $\m{A} \in\R^{200\times 400}$ and all elements of $\boldsymbol{n}$ are independently generated from a zero-mean Gaussian distribution. The original vector $\boldsymbol{x}^*$ is generated with approximately $2$\% nonzero entries.

In this numerical experiment, we set $L= \max_{i,j} |(\m{A}^\top \m{A})_{ij}|$, $\lambda = 2$, and $\varepsilon_1 = \varepsilon_2 = 10^{-6}$ in Algorithm~\ref{alg}. For the parameters in GPG, we choose the default values in the paper~\cite{xiao2022geometric} except $\lambda_0 = 0.01$, $\texttt{ITmax} = 3000$, and $\texttt{Tol} = 10^{-4}$. 
Figure~\ref{SNR result} displays RSNR and the required time of Algorithm \ref{alg} and GPG~\cite{xiao2022geometric} for different SNRs in the data $\boldsymbol{b}$. It shows that Algorithm~\ref{alg} achieves higher accuracy in a much shorter time than GPG. In addition, Algorithm~\ref{alg} would be more robust than GPG since it attains a lower standard error. 



\end{experiment}

\begin{experiment}
\label{exp: support accuracy}
\textbf{(support accuracy)} 
We choose a similar setup as in Experiment~\ref{SNR test}. We test on different sizes for the matrix $\m{A}$ under a fixed SNR = $50$. The original vector $\boldsymbol{x}^*$ is generated with approximately 4\% nonzero entries.

To quantitatively assess the performance of the algorithms in recovering the support of $\boldsymbol{x}^*$, we employ a confusion matrix with a detailed definition in Table~\ref{confusion matrix} and compute various metrics, including accuracy, precision, recall, and the F1 score. 

        
\begin{table}[ht]
	\centering
	\begin{tabular}{c|c|c|c|c}

        \multicolumn{2}{c}{\multirow{2}*{~}} & \multicolumn{2}{c}{\textbf{Actual}} & \\
        \hhline{~~--~}
        \multicolumn{2}{c|}{~} & Nonzero & Zero & \\
        \hhline{~|----}
        \multirow{3}{*}{\textbf{Predicted}} & Nonzero & \cellcolor{blues3} \makecell{True positive \quad \\(\textbf{TP})} & \cellcolor{blues1}\makecell{False positive\ \ \\(\textbf{FP})} & \multicolumn{1}{c|}{\makecell{ \textbf{Precision}\\ $\frac{\text{TP}}{\text{TP}+\text{FP}}$}}\\
        \hhline{~|----}
        & Zero & \cellcolor{blues1} \makecell{ False negative\\ (\textbf{FN})} & \cellcolor{blues3} \makecell{ True negative\\ (\textbf{TN})} & \\
    	\hhline{~|----}
        \multicolumn{2}{c|}{~} & \makecell{ \textbf{Recall}\\ $\frac{\text{TP}}{\text{TP}+\text{FN}}$} & & \multicolumn{1}{c|}{\makecell{ \textbf{Accuracy}\\ $\frac{\text{TP}+\text{TN}}{\text{TP}+\text{FN}+\text{FP}+\text{TN}}$}} \\
        \hhline{~~-~-}
        \noalign{\vspace*{0.5cm}}
        \multicolumn{5}{c}{$\textbf{F1} =\displaystyle {2\times \textbf{Precision}\times\textbf{Recall}\over \textbf{Precision}+\textbf{Recall}}$}
	\end{tabular}
    \caption{Confusion matrix metrics}
    \label{confusion matrix}
\end{table}

In our numerical experiment, we set $\varepsilon_1 = \varepsilon_2 = 10^{-7}$ in Algorithm~\ref{alg}. For the parameters in GPG, we choose a similar setup as in Experiment~\ref{SNR test} except for $\texttt{ITmax} = 2000$. For the regularization parameters $\lambda$ in Algorithm~\ref{alg} and $\lambda_0$ in GPG~\cite{xiao2022geometric}, we choose the values such that the predicted number of the nonzero elements in the estimated $\hat{\bsx}$ is approximately equal to the actual number of the nonzero elements in $\bsx^*$.

The comparative evaluation of the two algorithms is presented in Table~\ref{metrics}. The results indicate that both algorithms exhibit a high accuracy, signifying their ability to find sparse solutions, given that most elements in $\boldsymbol{x}^*$ are zero. Nevertheless, Algorithm~\ref{alg} demonstrates significantly enhanced precision, recall, and F1 scores compared to GPG. Specifically, for the nonzero elements within $\boldsymbol{x}^*$, GPG recovers only half of them, while our algorithm achieves nearly perfect recovery according to their recall and precision. In addition, Algorithm~\ref{alg} attains substantially lower objective function values within a shorter time.


\begin{table}
	\centering
	\begin{tabular}{c|c|cccccc}
		\hline
         \multicolumn{2}{c|}{~}   & accuracy & precision & recall & F1 & time (s) & $\frac{1}{2}\|\m{A} \hat{\boldsymbol{x}} - \boldsymbol{b}\|^2$ \\
        \hline
        \multirow{2}*{I} &  Alg.~\ref{alg} & \textbf{0.994} & \textbf{0.969} & \textbf{0.939} & \textbf{0.949} & \textbf{0.017} & $\mathbf{6.50 \times 10^{-4}}$ \\
        ~ & GPG~\cite{xiao2022geometric} &  0.962 & 0.556 & 0.504 & 0.520 & 0.088 & $2.578 \times 10^{-1}$\\
		\hline
        \multirow{2}*{II} & Alg.~\ref{alg} &  \textbf{0.999} & \textbf{0.990} & \textbf{0.988} & \textbf{0.989} & \textbf{0.241} & $\mathbf{2.188 \times 10^{-5}}$ \\
        ~ & GPG~\cite{xiao2022geometric} & 0.961 & 0.511 & 0.496 & 0.501 & 1.778 & $3.656 \times 10^{-1}$\\
        \hline
	\end{tabular}
    \caption{Performance based on different metrics under case I ($\m{A}_{50 \times 300}$) and case II ($\m{A}_{170 \times 900}$) (average over 100 runs).}
    \label{metrics}
\end{table}

The efficacy of Algorithm~\ref{alg} is vividly depicted in Figure~\ref{A300A900}, which presents a comprehensive overview of the metric performances across various dimensions of the matrix
$\m{A}$. Irrespective of the matrix size, Algorithm~\ref{alg} consistently exhibits high accuracy, suggesting its robust ability to identify TP and TN with remarkable accuracy. In addition, a noteworthy observation is the upward trend in all metrics as the row dimension of the matrix $\m{A}$ expands, especially for the precision, recall, and F1. 
As the ratio of the number of rows to the number of columns increases, Algorithm~\ref{alg} progressively exhibits a remarkable capability to recover the ground truth vector $\bsx^*$.


\begin{figure}
\centering
  \begin{subfigure}[t]{0.45\linewidth}
    \centering
    \includegraphics[scale=0.3]{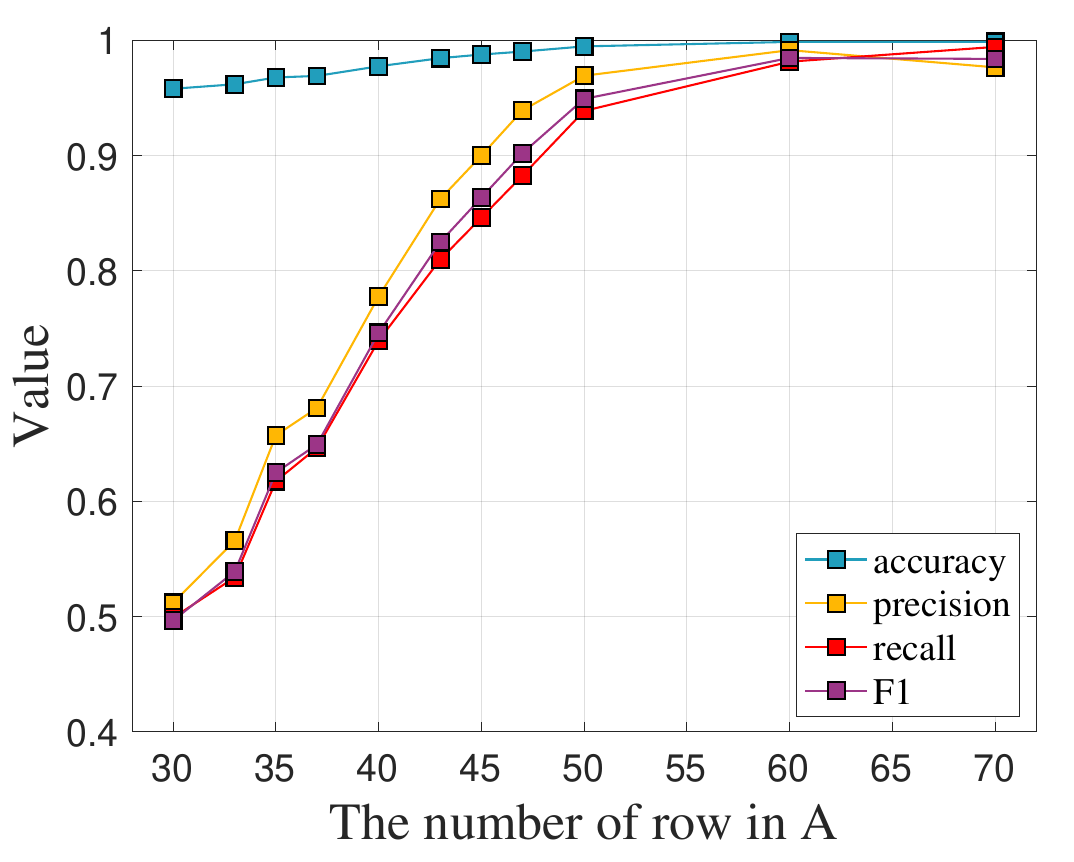}
  \end{subfigure}%
  \hspace{10mm}
  \begin{subfigure}[t]{0.45\linewidth}
    \centering
    \includegraphics[scale=0.3]{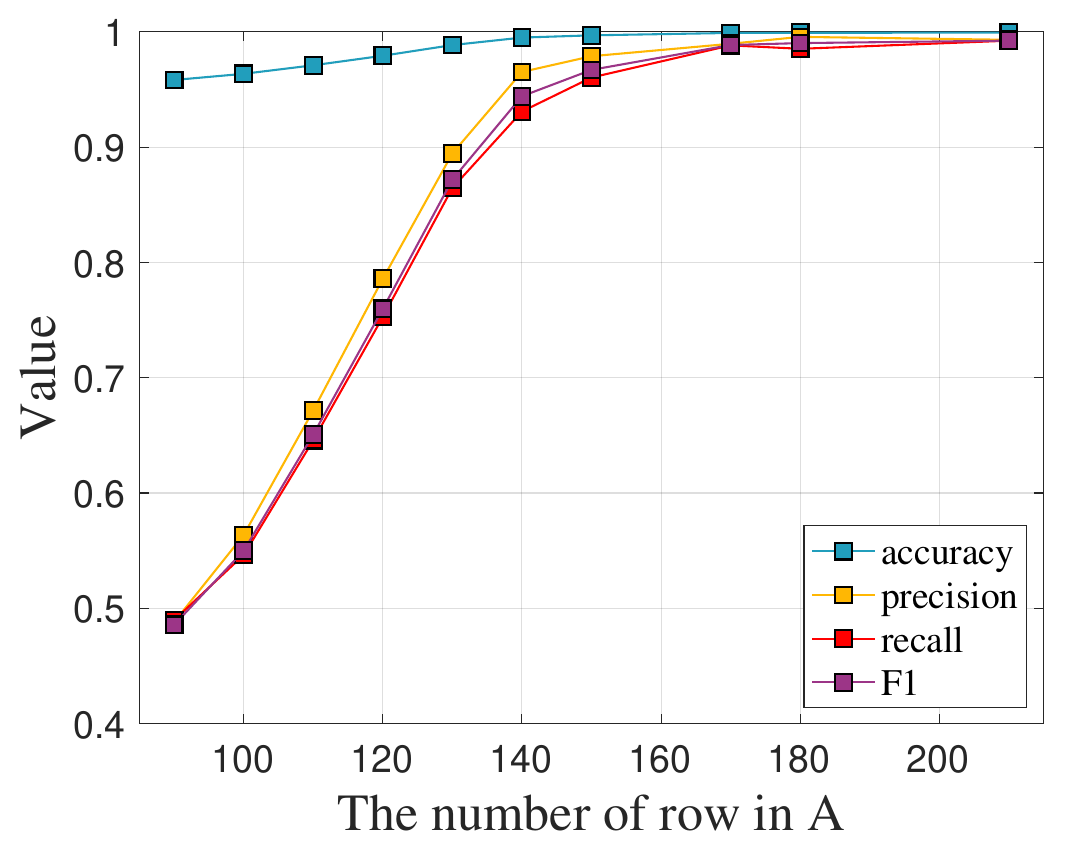}
  \end{subfigure}
  \caption{(Left) Metrics performance under fixed 300 columns in $\m{A}$ (average over 100 runs). (Right) Metrics performance under fixed 900 columns in $\m{A}$ (average over 100 runs)}
  \label{A300A900}
\end{figure}

\end{experiment}

\begin{experiment}
\label{exp: efficacy L0BPG step}
(\textbf{{\color{blue}e}fficacy of the L0BPG step in Algorithm~\ref{alg}}) We validate the efficacy of Algorithm~\ref{alg} in picking an element that was not ranked high during the initialization phase. We choose a similar setup as in Experiment~\ref{SNR test} with $\m{A} \in \R^{60 \times 300}$ without noise. 
The original vector $\bsx^*$ has $15$ nonzero entries, and we will get a vector with $N (<15)$ nonzero elements with the sparsity penalty.

We set $\lambda = 1.5$, $N = 12$ and $\varepsilon_1 = \varepsilon_2 = 10^{-8}$. We plot the numerical changes of four elements (12th to 15th largest components in $\bsx^*$) in Figure~\ref{fig:numerical change}. Note that we rank the four elements based on the original vector $\bsx^*$, and the figure shows that Algorithm~\ref{alg} picks the original 14th largest element instead of the 12th one. 

\begin{figure}[t]
    \begin{minipage}[p]{0.45\textwidth} 
    \includegraphics[scale=0.3]{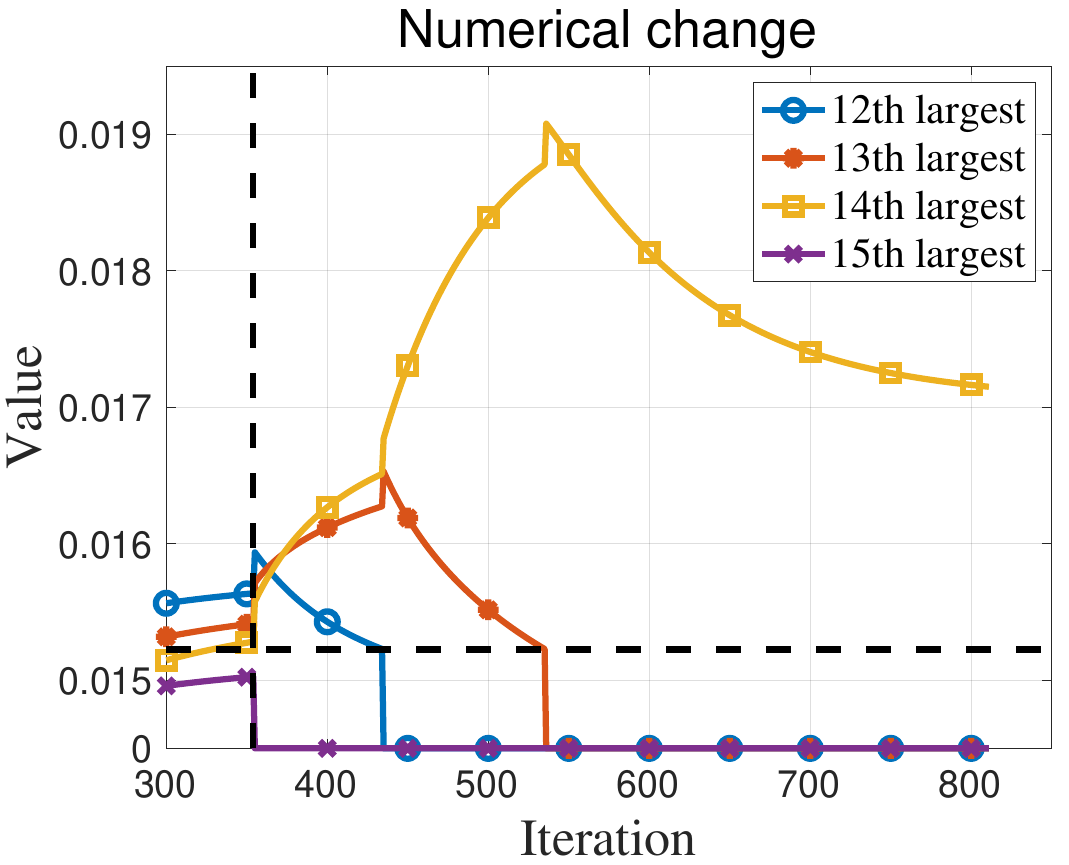}
    \end{minipage}
    \hspace{5mm}
    \begin{minipage}[c]{0.45\textwidth}
    \centering
    \begin{tabular}{|c|c|c|}
        \hline
        ~ & \makecell{Alg.~\ref{alg} \\ keep 14th}  & keep 12th\\
        \hline
        $\frac{1}{2}\|\m{A} \hat{\bsx} - \boldsymbol{b}\|^2$ & \textbf{0.0176} & 0.0191\\
        \hline
    \end{tabular}
    \end{minipage}
    \caption{(Left) The numerical changes of the four elements ranked from the 12th to the 15th largest of the original $\bsx^*$ in Algorithm~\ref{alg}. The dashed vertical and horizon lines indicate the completion of the initialization phase and the minimal value controlled, i.e., $1-e^{-\alpha \lambda}$, respectively. (Right) The error comparison for including the 14th and the 12th elements of $\bsx^*$, where $\hat{\bsx}$ is the output vector. Both vectors contain the first 11 largest elements in $\bsx^*$. }
    \label{fig:numerical change}
\end{figure}

Given that Algorithm~\ref{alg} outputs a vector with 12 nonzero elements, we further substantiate that the 14th largest element in $\bsx^*$ outperforms the 12th by solving the problem without the sparsity penalty, yet preserving only the top 12 largest elements in $\bsx^*$. Both vectors contain the first 11 largest elements in $\bsx^*$, and Figure~\ref{fig:numerical change} shows that including the 14th one has a smaller error than that including the 12th one, which validates the efficacy of Algorithm~\ref{alg}.
\end{experiment}

\begin{experiment}
\label{exp: huber loss}
(\textbf{Huber loss}) We consider $\phi$ as the Huber loss function, which is defined as
\begin{equation}
    \phi(e) = \left\{  
         \begin{aligned}
         & \frac{1}{2}e^2, \quad |e| \leq c\\ 
         & c|e| - \frac{1}{2}c^2, \quad  |e| > c,
         \end{aligned}  
         \right.
\end{equation}
where $c$ is a cutoff parameter controlling the level of robustness. 

We choose $c=1.0$ and $\m{A} \in \R^{200 \times 400}$. The original vector $\boldsymbol{x}^*$ is generated with approximately 2\% nonzero entries. Contrary to the noise $\boldsymbol{n}$ generated in Experiment~\ref{SNR test}, we consider the noise $\boldsymbol{n}$ is comprised of Gaussian noise $\boldsymbol{n}_G$ and impulse noise $\boldsymbol{n}_I$:
\begin{itemize}
    \item Gaussian noise $\boldsymbol{n}_G$: we generate the noise under a fixed SNR = 20, where 
    $
    SNR = 10\log_{10} (\left\| \m{A} \boldsymbol{x}^* \right\|^2)/(\|\boldsymbol{n}_G\|^2).
    $
    \item Salt and Pepper Impulse noise $\boldsymbol{n}_I$: we generate the noise using \texttt{imnoise} from Matlab to data $\boldsymbol{b}$ under different densities  ranging from 0.1 to 0.2 with values 0 or $20\|\boldsymbol{n}_G\|_{\infty}$. 
\end{itemize}

For parameter choosing, we set $L=\max_{i,j}|(\m{A}^\top\m{A})_{ij}|$, $\varepsilon_1 = \varepsilon_2 = 10^{-6}$ in Algorithm~\ref{alg}. Figure~\ref{fig: Huber loss} presents RSNR for both Huber loss and quadratic loss using our Algorithm~\ref{alg}. It can be observed that RSNR decreases when the impulse noise density increases. With the mixed noise, using Huber loss with different $\lambda$ values consistently achieves a higher accuracy than using the quadratic loss. 

\begin{figure}[H]
    \centering
    \includegraphics[scale = 0.4]{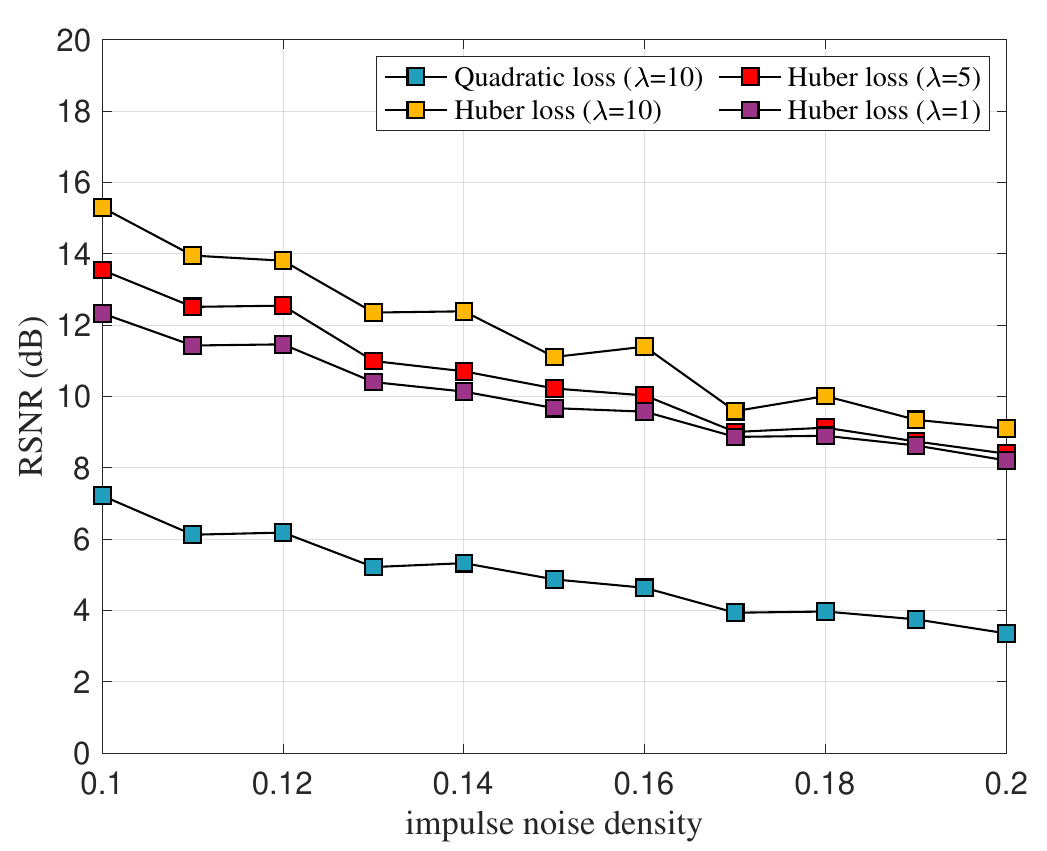}
    \caption{Comparison of Huber loss and quadratic loss in accuracy (RSNR) using Algorithm~\ref{alg} with different $\lambda$ (averaged over 100 runs).}
    \label{fig: Huber loss}
\end{figure}
\end{experiment}

\subsection{Experiments with real data}
\label{sec: real}
We consider two specific optimization problems: the sparse hyperspectral unmixing problem and the sparse portfolio optimization problem.
\begin{experiment}
(\textbf{hyperspectral unmixing}) We consider the following optimization problem: 
\begin{equation}
\begin{aligned}
\label{opt: hyperspectral unmixing}
    \min_{\m{X} \in \R^{n \times p}} \quad & \frac{1}{2} \|\m{A}\m{X} - \m{B}\|_F^2 + \lambda \|\m{X}\|_0\\
    \text{subject to } \quad  & \mathbf{1}_n^\top \m{X} = \mathbf{1}_p^\top,\ \m{X} \geq 0,
\end{aligned}
\end{equation}
where $\m{B} \in \R^{m \times p}$ represents the observed hyperspectral image with $m$ bands and $p$ pixels, $\m{A} \in \R^{m \times n}$ is the spectral library composed of $n$ endmembers, and $\m{X} \geq 0$ indicates that all elements in $\m{X}$ are nonnegative. The primary objective of hyperspectral unmixing is to recover the unknown abundance matrix $\m{X}$ and estimate the spatial distributions or relative proportions of spectral signatures within each pixel~\cite{ince2022fast}. It's worth noting that, given the likelihood of only a few spectral signatures from $\m{A}$ contributing to the observed spectra of each pixel, each column of the matrix $\m{X}$ is typically sparse.

In this experiment, we focus on a well-known region of the Cuprite dataset\footnote{\url{https://aviris.jpl.nasa.gov/data/free_data.html}} with 250×191 pixels. The original hyperspectral image comprises 224 bands; however, we excluded some bands due to their low SNR, resulting in 188 bands. The spectral library $\m{A}$ was built as a collection of $n = 498$ spectral signatures in the USGS library. Given that the true abundance maps of the Cuprite dataset are not available, we will refer to the Geological Reference Map~\cite{rasti2024image} and the estimated abundance matrix generated by SUnSAL algorithm~\cite{bioucas2010alternating}, whose performance was evaluated and compared to the Tricorder maps in~\cite{iordache2009unmixing}. 

\begin{figure}[H]
\centering
    \begin{minipage}[c]{0.45\textwidth} 
    \includegraphics[scale=0.9]{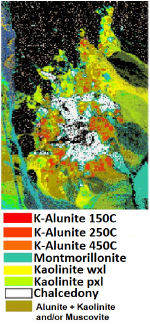}
    \end{minipage}
    \hspace{-1cm}
    \begin{minipage}[c]{0.45\textwidth}
    \includegraphics[scale=0.35]{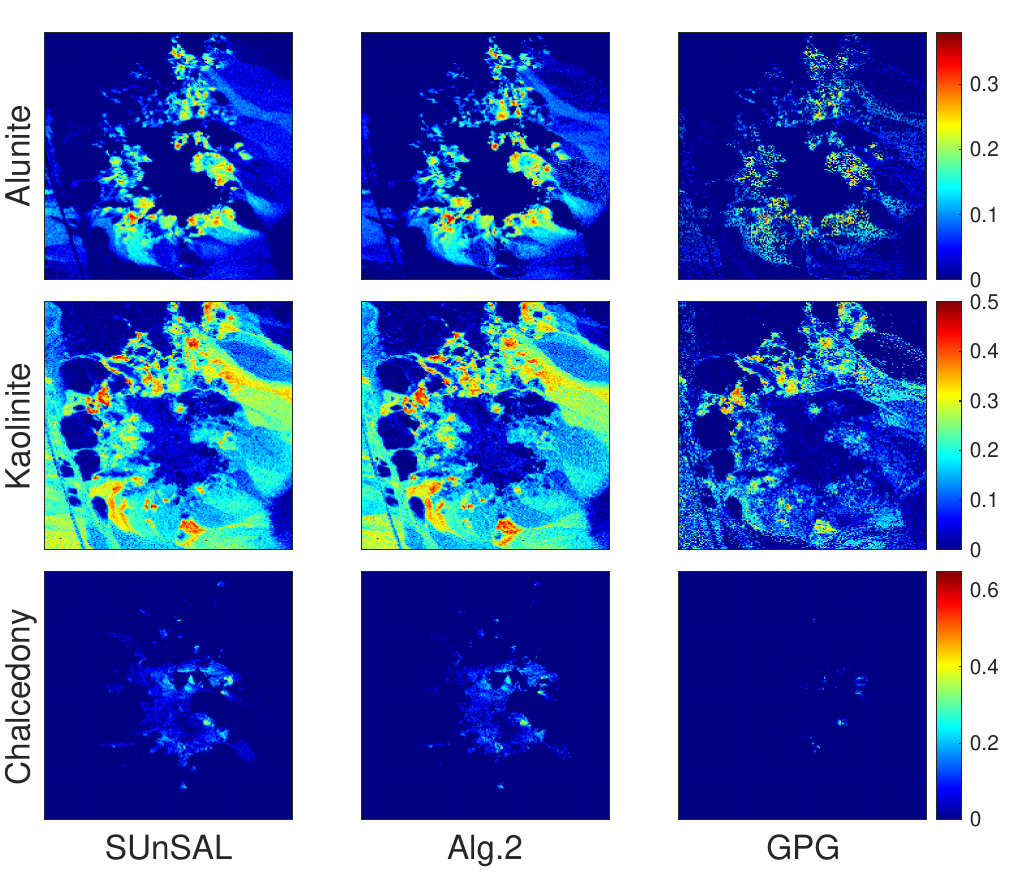}
    \end{minipage}
    \caption{(Left) Geological Reference Map~\cite{rasti2024image}. (Right) Estimated abundance fraction maps using different methods. }
    \label{fig: hyperspectral unmixing}
\end{figure}

We set $\lambda = 5$, $\varepsilon_1 = \varepsilon_2 = 10^{-6}$ in Algorithm~\ref{alg}, $\lambda = 10^{-3}$ in SUnSAL, and $\lambda_0 = 10^{-3}$, $\texttt{ITmax} = 2000$, $\gamma_2 = 10^{-4}$, $\rho_2 = 0.5$, $\texttt{Tol} = 10^{-4}$ in GPG. Figure~\ref{fig: hyperspectral unmixing} illustrates the abundance maps for three representative minerals: alunite, kaolinite, and chalcedony. We can infer that our Algorithm~\ref{alg} exhibits high similarity to SUnSAL and the Geological Reference Map, while $\m{X}^*$ generated by GPG appears to be noise. In addition, it is worth mentioning that our Algorithm~\ref{alg} generates the most sparse solution, as outlined in Table~\ref{tab: summary of estimated X*}. 

\begin{table}[t]
\centering
    \begin{tabular}{|c|c|c|c|}
        \hline
        \multirow{2}*{Summary} & \multicolumn{3}{c|}{Estimated abundance matrix $\m{X}^*$} \\
        \cline{2-4}
        ~ & SUnSAL & Alg.~\ref{alg}  & GPG\\
        \hline
        $\min_i \{\|\boldsymbol{x}^*_i\|_0\}$ & 10 & \textbf{2} & 19\\
        \hline
        $\max_i \{\|\boldsymbol{x}^*_i\|_0\}$ & 32 & \textbf{18} & 494\\
        \hline
        $\frac{1}{n} \sum_{i=1}^m \|\boldsymbol{x}^*_i\|_0$ & 18.031 & \textbf{9.0989} & 191.63\\
        \hline
    \end{tabular}
    \caption{Summary of estimated $\m{X}^*$ using different methods.}
    \label{tab: summary of estimated X*}
\end{table}

\end{experiment}

\begin{experiment}
(\textbf{portfolio optimization}) We consider the following optimization problem:
\begin{equation}
\begin{aligned}
\label{portfolio opt}
    \min_{\boldsymbol{x}\in\R^n} \quad & \frac{1}{2} \eta \left( \boldsymbol{x}^\top \m{\Sigma} \boldsymbol{x}\right) - (1-\eta)\left(\boldsymbol{\mu}^\top \boldsymbol{x}\right)\\
    \text{subject to } \quad  & {\color{blue}\mathbf{1}_n}^\top \boldsymbol{x} = 1, \boldsymbol{x} \geq 0,
\end{aligned}
\end{equation}
where $\boldsymbol{\mu}$ and $\m{\Sigma}$ are the expected return and the covariance matrix of the $n$ assets, respectively. Portfolio optimization aims to maximize the expected return ($\boldsymbol{\mu}^\top \boldsymbol{x}$) while minimizing the risk ($\boldsymbol{x}^\top \m{\Sigma} \boldsymbol{x}$), and the risk aversion parameter $\eta$ balances both objectives. By varying $\eta\in[0,1]$, the optimization problem returns different portfolios that form the efficient frontier in the context of Markowitz's theory~\cite{markowits1952portfolio}. There are mainly two types of efficient frontiers~\cite{fernandez2007portfolio}: the standard efficient frontier, which solves the problem~\eqref{portfolio opt}, and the general efficient frontier, which solves the same problem with an additional sparsity constraint $\|\bsx\|_0 \leq K$ for a given $K$.

We use the benchmark datasets for portfolio optimization from the OR-Library\footnote{\url{http://people.brunel.ac.uk/~mastjjb/jeb/orlib/portinfo.html}}. The datasets have weekly prices of some assets from five financial markets (Hang Seng in Hong Kong, DAX 100 in Germany, FTSE 100 in the UK, S\&P 100 in the USA, and Nikkei 225 in Japan) between March 1992 and September 1997. The numbers of assets in the five markets were 31, 85, 89, 98, and 225, respectively. We choose 2000 and 50 evenly spaced $\eta$ values for the standard efficient frontiers (SEF) and the general efficient frontiers (GEF), respectively. We also set $K = 10$ in the general efficient frontier. To compare these two frontiers, we employ three criteria: mean Euclidean distance, variance of return (risk) error, and mean return error~\cite{yin2015novel,cura2009particle}. These three metrics describe the overall distance between these two frontiers, the risk's relative error, and the mean return's relative error. They are referred to as distance, variance, and mean, respectively, in Table \ref{portfolio}.

Table \ref{portfolio} demonstrates that the general efficient frontier closely approximates the standard efficient frontier, as indicated by the low values of mean Euclidean distance, variance of return error, and mean return error. We also plot the general and standard efficient frontier in Figure~\ref{frontiers}, which visualizes the distance between these two frontiers. In these figures, the points with sparsity 10 are still very close to the curve without sparsity. 

\begin{table}
	\centering
	\begin{tabular}{cccccc}
		\hline
        index & Hang Seng & DAX 100 & FTSE 100 & S\&P 100 & Nikkei \\
        \hline
        assets & 31 & 85 & 89 & 98 & 225 \\
        distance ($\times 10^{-6}$) & $1.683$ & $1.311$ & $1.269$ & $9.448$ & $1.583$\\
        variance (\%) & 0.058 & 0.251 & 0.248 & 0.637 & 0.043\\
        mean (\%) & 0.0263 & 0.027 & 0.025 & 0.527 & 1.970\\
        time (s) & 0.397 & 0.549 & 0.509 & 0.720 & 13.761\\
        \hline
	\end{tabular}
    \caption{The numerical results for sparse portfolio optimization problem. The last row is the total time for calculating these two frontiers.}
    \label{portfolio}
\end{table}

\begin{figure}
\centering
  \begin{subfigure}[t]{0.3\linewidth}
    \centering
    \includegraphics[scale=0.2]{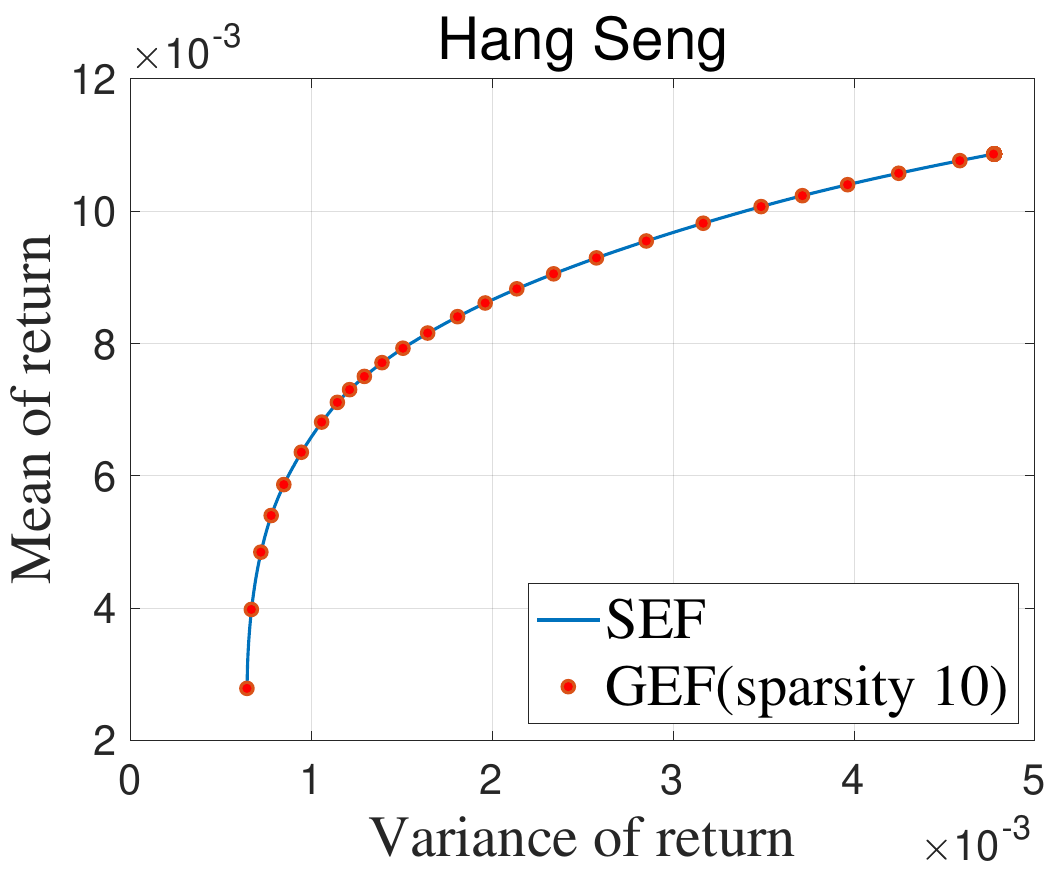}
    \caption{}
    \label{HangSeng}
  \end{subfigure}%
  \hspace{20mm}
  \begin{subfigure}[t]{0.3\linewidth}
    \centering
    \includegraphics[scale=0.2]{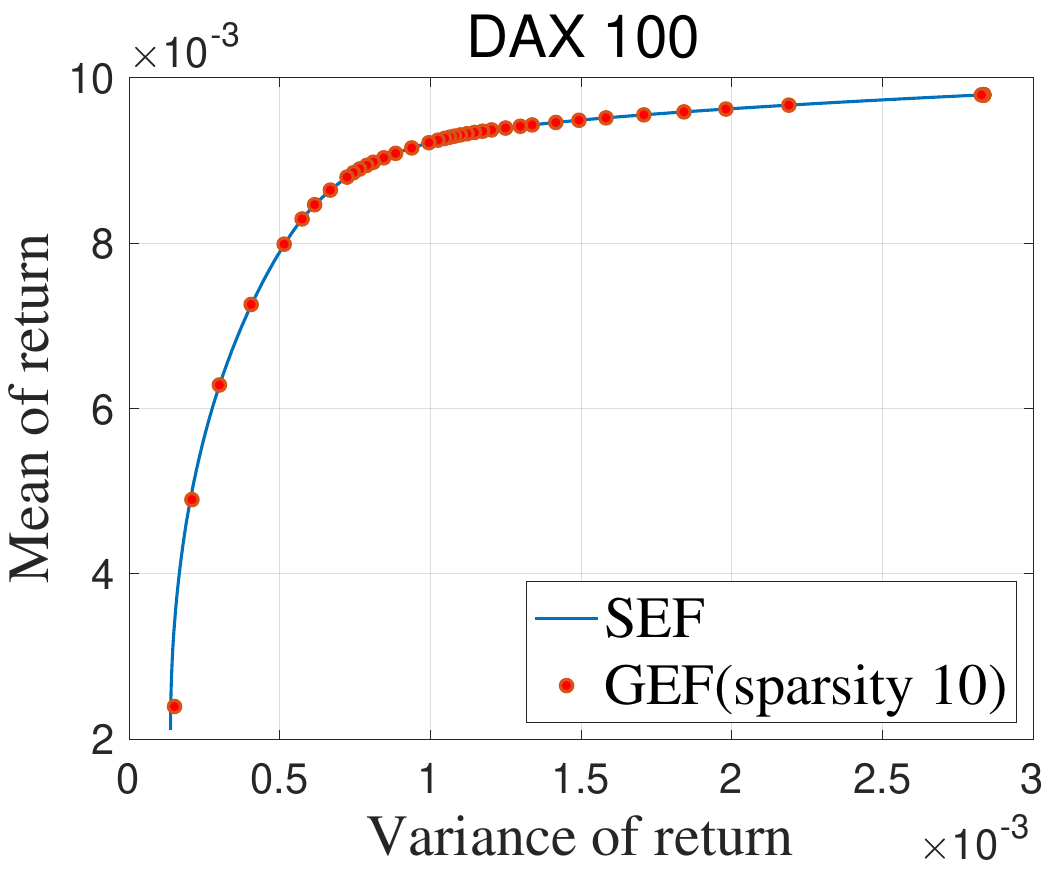}
    \caption{}
    \label{DAX100}
  \end{subfigure}    
    
    \quad
   \begin{subfigure}[t]{0.3\linewidth}
    \centering
    \includegraphics[scale=0.2]{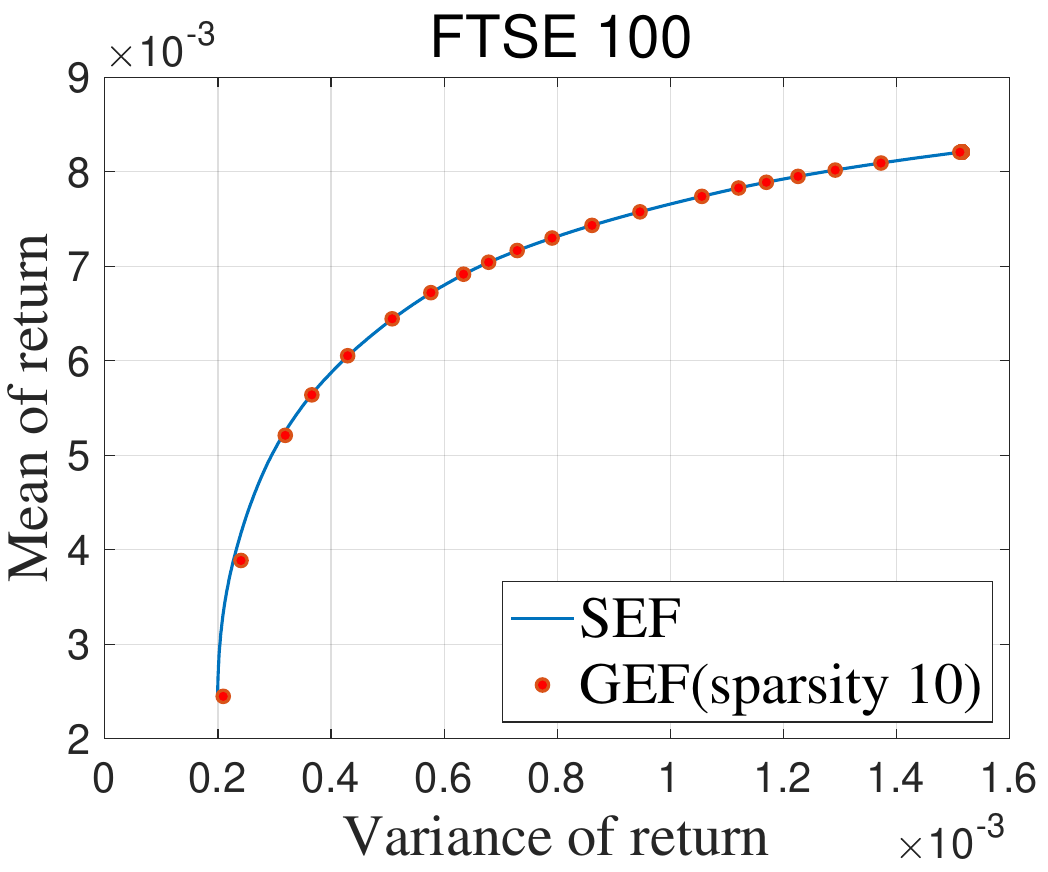}
    \caption{}
    \label{FTSE100}
  \end{subfigure}%
  \hspace{1mm}
  \begin{subfigure}[t]{0.3\linewidth}
    \centering
    \includegraphics[scale=0.2]{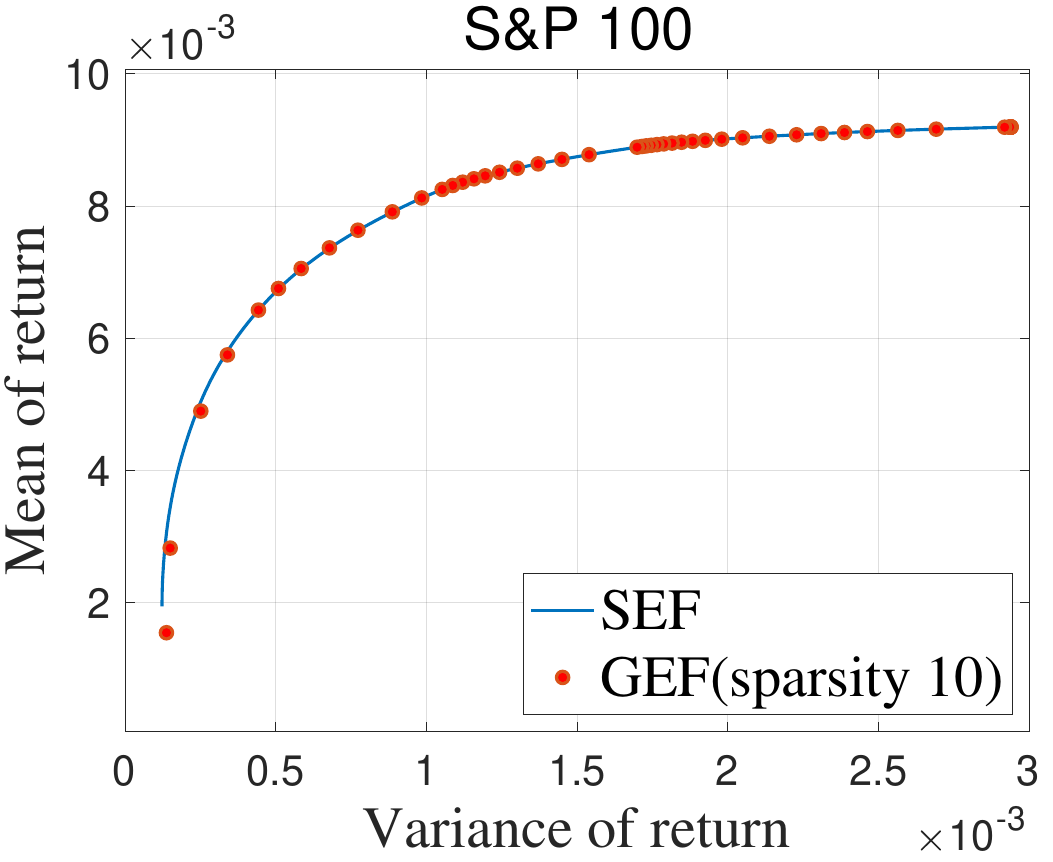}
    \caption{}
    \label{SP100}
  \end{subfigure}
  \hspace{1mm}
  \begin{subfigure}[t]{0.3\linewidth}
    \centering
    \includegraphics[scale=0.2]{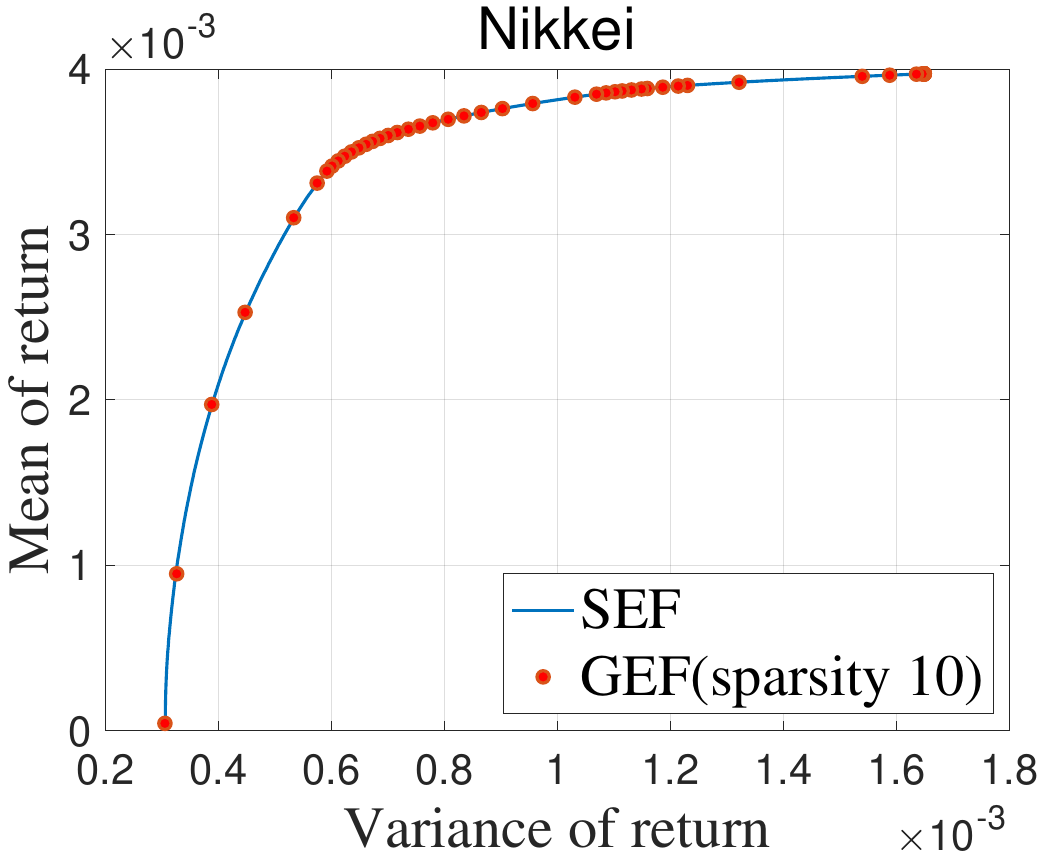}
    \caption{}
    \label{Nikkei}
  \end{subfigure}
\caption{(a) Efficient frontier for Hang Seng. (b) Efficient frontier for DAX 100. (c) Efficient frontier for FTSE 100. (d) Efficient frontier for S\&P 100. (e) Efficient frontier for Nikkei.}
\label{frontiers}
\end{figure}

\end{experiment}

\section{Conclusion}
This paper addresses the $\ell_0$-sparse optimization problem subject to a probability simplex constraint. We introduce an innovative algorithm that leverages the Bregman proximal gradient method to progressively induce sparsity by explicitly solving the associated subproblems. Our work includes a rigorous convergence analysis of this proposed algorithm, demonstrating its capability to reach a local minimum with a convergence rate of $O(1/k)$. Additionally, the empirical results illustrate the superior performance of the proposed algorithm. Finally, Future work will delve into strategies for reintroducing important elements that have been set to zero during the algorithmic process and the design of an adaptive regularized parameter.



\section*{Declarations}
\textbf{Acknowledgements} This work was partially supported by the Guangdong Key Laboratory of Mathematical Foundations for Artificial Intelligence 2023B1212010001 and Shenzhen Science and Technology Program ZDSYS20211021111415025. \\
\\
\textbf{Data availability} The data generated in Subsection~\ref{sec: synthetic} are available at \url{https://github.com/PanT12/Efficient-sparse-probability-measures-recovery-via-Bregman-gradient} . The data that support Subsection~\ref{sec: real} are openly available in \url{https://aviris.jpl.nasa.gov/data/free_data.html} and 
\url{http://people.brunel.ac.uk/~mastjjb/jeb/orlib/portinfo.html}.\\
\textbf{Conflict of interest} The authors declare that they have no conflict of interest.

\bibliographystyle{spmpsci}      

\begin{thebibliography}{10}
\providecommand{\url}[1]{{#1}}
\providecommand{\urlprefix}{URL }
\expandafter\ifx\csname urlstyle\endcsname\relax
  \providecommand{\doi}[1]{DOI~\discretionary{}{}{}#1}\else
  \providecommand{\doi}{DOI~\discretionary{}{}{}\begingroup
  \urlstyle{rm}\Url}\fi

\bibitem{auslender2006interior}
Auslender, A., Teboulle, M.: Interior gradient and proximal methods for convex
  and conic optimization.
\newblock SIAM Journal on Optimization \textbf{16}(3), 697--725 (2006)

\bibitem{bauschke2017descent}
Bauschke, H.H., Bolte, J., Teboulle, M.: A descent lemma beyond {L}ipschitz
  gradient continuity: first-order methods revisited and applications.
\newblock Mathematics of Operations Research \textbf{42}(2), 330--348 (2017)

\bibitem{beck2003mirror}
Beck, A., Teboulle, M.: Mirror descent and nonlinear projected subgradient
  methods for convex optimization.
\newblock Operations Research Letters \textbf{31}(3), 167--175 (2003)

\bibitem{ben2001ordered}
Ben-Tal, A., Margalit, T., Nemirovski, A.: The ordered subsets mirror descent
  optimization method with applications to tomography.
\newblock SIAM Journal on Optimization \textbf{12}(1), 79--108 (2001)

\bibitem{bertsimas2022scalable}
Bertsimas, D., Cory-Wright, R.: A scalable algorithm for sparse portfolio
  selection.
\newblock Informs journal on computing \textbf{34}(3), 1489--1511 (2022)

\bibitem{bioucas2010alternating}
Bioucas-Dias, J.M., Figueiredo, M.A.: Alternating direction algorithms for
  constrained sparse regression: Application to hyperspectral unmixing.
\newblock In: 2010 2nd Workshop on Hyperspectral Image and Signal Processing:
  Evolution in Remote Sensing, pp. 1--4. IEEE (2010)

\bibitem{birnbaum2011distributed}
Birnbaum, B., Devanur, N.R., Xiao, L.: Distributed algorithms via gradient
  descent for {F}isher markets.
\newblock In: Proceedings of the 12th ACM conference on Electronic commerce,
  pp. 127--136 (2011)

\bibitem{blumensath2008iterative}
Blumensath, T., Davies, M.E.: Iterative thresholding for sparse approximations.
\newblock Journal of Fourier analysis and Applications \textbf{14}, 629--654
  (2008)

\bibitem{blumensath2009iterative}
Blumensath, T., Davies, M.E.: Iterative hard thresholding for compressed
  sensing.
\newblock Applied and computational harmonic analysis \textbf{27}(3), 265--274
  (2009)

\bibitem{blumensath2010normalized}
Blumensath, T., Davies, M.E.: Normalized iterative hard thresholding:
  {G}uaranteed stability and performance.
\newblock IEEE Journal of selected topics in signal processing \textbf{4}(2),
  298--309 (2010)

\bibitem{bolte2018first}
Bolte, J., Sabach, S., Teboulle, M., Vaisbourd, Y.: First order methods beyond
  convexity and {L}ipschitz gradient continuity with applications to quadratic
  inverse problems.
\newblock SIAM Journal on Optimization \textbf{28}(3), 2131--2151 (2018)

\bibitem{boyd2004convex}
Boyd, S.P., Vandenberghe, L.: Convex optimization.
\newblock Cambridge university press (2004)

\bibitem{cura2009particle}
Cura, T.: Particle swarm optimization approach to portfolio optimization.
\newblock Nonlinear analysis: Real world applications \textbf{10}(4),
  2396--2406 (2009)

\bibitem{eckstein1993nonlinear}
Eckstein, J.: Nonlinear proximal point algorithms using {B}regman functions,
  with applications to convex programming.
\newblock Mathematics of Operations Research \textbf{18}(1), 202--226 (1993)

\bibitem{esmaeili2016l}
Esmaeili~Salehani, Y., Gazor, S., Kim, I.M., Yousefi, S.: $\ell_0$-norm sparse
  hyperspectral unmixing using arctan smoothing.
\newblock Remote Sensing \textbf{8}(3), 187 (2016)

\bibitem{fernandez2007portfolio}
Fern{\'a}ndez, A., G{\'o}mez, S.: Portfolio selection using neural networks.
\newblock Computers \& operations research \textbf{34}(4), 1177--1191 (2007)

\bibitem{fornasier2015compressive}
Fornasier, M., Rauhut, H.: Compressive sensing.
\newblock Handbook of mathematical methods in imaging \textbf{1}, 187--229
  (2015)

\bibitem{guo2021modified}
Guo, Z., Min, A., Yang, B., Chen, J., Li, H.: A modified huber nonnegative
  matrix factorization algorithm for hyperspectral unmixing.
\newblock IEEE Journal of Selected Topics in Applied Earth Observations and
  Remote Sensing \textbf{14}, 5559--5571 (2021)

\bibitem{hanzely2021accelerated}
Hanzely, F., Richtarik, P., Xiao, L.: Accelerated {B}regman proximal gradient
  methods for relatively smooth convex optimization.
\newblock Computational Optimization and Applications \textbf{79}, 405--440
  (2021)

\bibitem{ince2022fast}
Ince, T., Dobigeon, N.: Fast hyperspectral unmixing using a multiscale sparse
  regularization.
\newblock IEEE Geoscience and Remote Sensing Letters \textbf{19}, 1--5 (2022)

\bibitem{iordache2009unmixing}
Iordache, M.D., Bioucas-Dias, J., Plaza, A.: Unmixing sparse hyperspectral
  mixtures.
\newblock In: 2009 IEEE International Geoscience and Remote Sensing Symposium,
  vol.~4, pp. IV--85. IEEE (2009)

\bibitem{jiang2023bregman}
Jiang, X., Vandenberghe, L.: Bregman three-operator splitting methods.
\newblock Journal of Optimization Theory and Applications \textbf{196}(3),
  936--972 (2023)

\bibitem{krichene2015accelerated}
Krichene, W., Bayen, A., Bartlett, P.L.: Accelerated mirror descent in
  continuous and discrete time.
\newblock Advances in neural information processing systems \textbf{28} (2015)

\bibitem{lu2018relatively}
Lu, H., Freund, R.M., Nesterov, Y.: Relatively smooth convex optimization by
  first-order methods, and applications.
\newblock SIAM Journal on Optimization \textbf{28}(1), 333--354 (2018)

\bibitem{ma2011fixed}
Ma, S., Goldfarb, D., Chen, L.: Fixed point and {B}regman iterative methods for
  matrix rank minimization.
\newblock Mathematical Programming \textbf{128}(1-2), 321--353 (2011)

\bibitem{majumdar2016impulse}
Majumdar, A., Ansari, N., Aggarwal, H., Biyani, P.: Impulse denoising for
  hyper-spectral images: A blind compressed sensing approach.
\newblock Signal Processing \textbf{119}, 136--141 (2016)

\bibitem{markowits1952portfolio}
Markowits, H.M.: Portfolio selection.
\newblock Journal of finance \textbf{7}(1), 71--91 (1952)

\bibitem{natarajan1995sparse}
Natarajan, B.K.: Sparse approximate solutions to linear systems.
\newblock SIAM journal on computing \textbf{24}(2), 227--234 (1995)

\bibitem{nemirovskij1983problem}
Nemirovskij, A.S., Yudin, D.B.: Problem complexity and method efficiency in
  optimization  (1983)

\bibitem{pan2017convergent}
Pan, L., Zhou, S., Xiu, N., Qi, H.D.: A convergent iterative hard thresholding
  for nonnegative sparsity optimization.
\newblock Pacific Journal of Optimization \textbf{13}(2), 325--353 (2017)

\bibitem{rasti2024image}
Rasti, B., Zouaoui, A., Mairal, J., Chanussot, J.: Image processing and machine
  learning for hyperspectral unmixing: An overview and the hysupp python
  package.
\newblock IEEE Transactions on Geoscience and Remote Sensing  (2024)

\bibitem{rogass2014reduction}
Rogass, C., Mielke, C., Scheffler, D., Boesche, N.K., Lausch, A., Lubitz, C.,
  Brell, M., Spengler, D., Eisele, A., Segl, K., et~al.: Reduction of
  uncorrelated striping noise—applications for hyperspectral pushbroom
  acquisitions.
\newblock Remote Sensing \textbf{6}(11), 11082--11106 (2014)

\bibitem{salehani2014sparse}
Salehani, Y.E., Gazor, S., Kim, I.M., Yousefi, S.: Sparse hyperspectral
  unmixing via arctan approximation of {L}0 norm.
\newblock In: 2014 IEEE Geoscience and Remote Sensing Symposium, pp.
  2930--2933. IEEE (2014)

\bibitem{tang2014sparse}
Tang, W., Shi, Z., Duren, Z.: Sparse hyperspectral unmixing using an
  approximate {L}0 norm.
\newblock Optik \textbf{125}(1), 31--38 (2014)

\bibitem{tibshirani1996regression}
Tibshirani, R.: Regression shrinkage and selection via the lasso.
\newblock Journal of the Royal Statistical Society Series B: Statistical
  Methodology \textbf{58}(1), 267--288 (1996)

\bibitem{xiao2022geometric}
Xiao, G., Bai, Z.J.: A geometric proximal gradient method for sparse least
  squares regression with probabilistic simplex constraint.
\newblock Journal of Scientific Computing \textbf{92}(1), 22 (2022)

\bibitem{xu2011image}
Xu, L., Lu, C., Xu, Y., Jia, J.: Image smoothing via {L}0 gradient
  minimization.
\newblock In: Proceedings of the 2011 SIGGRAPH Asia conference, pp. 1--12
  (2011)

\bibitem{yin2015novel}
Yin, X., Ni, Q., Zhai, Y.: A novel {PSO} for portfolio optimization based on
  heterogeneous multiple population strategy.
\newblock In: 2015 IEEE Congress on Evolutionary Computation (CEC), pp.
  1196--1203. IEEE (2015)

\bibitem{zhang2019learning}
Zhang, J.Y., Khanna, R., Kyrillidis, A., Koyejo, O.O.: Learning sparse
  distributions using iterative hard thresholding.
\newblock Advances in Neural Information Processing Systems \textbf{32} (2019)

\bibitem{zhang2023sparse}
Zhang, P., Xiu, N., Qi, H.D.: Sparse {SVM} with hard-margin loss: a
  {N}ewton-augmented lagrangian method in reduced dimensions.
\newblock arXiv preprint arXiv:2307.16281  (2023)

\bibitem{zhao2022lagrange}
Zhao, C., Xiu, N., Qi, H., Luo, Z.: A {L}agrange--{N}ewton algorithm for sparse
  nonlinear programming.
\newblock Mathematical Programming \textbf{195}(1-2), 903--928 (2022)

\bibitem{zhou2021global}
Zhou, S., Xiu, N., Qi, H.D.: Global and quadratic convergence of {N}ewton
  hard-thresholding pursuit.
\newblock The Journal of Machine Learning Research \textbf{22}(1), 599--643
  (2021)

\bibitem{zou2018restoration}
Zou, C., Xia, Y.: Restoration of hyperspectral image contaminated by poisson
  noise using spectral unmixing.
\newblock Neurocomputing \textbf{275}, 430--437 (2018)

\end{thebibliography}


\end{document}